\documentclass[11pt]{amsart}

\usepackage{amsmath}
\usepackage{amsfonts}
\usepackage{amssymb}
\usepackage{graphicx}
\usepackage{hyperref}
\usepackage{mathrsfs}
\usepackage{pb-diagram}
\usepackage{epstopdf}
\usepackage{amscd}
\usepackage{graphicx}
\usepackage{pb-diagram}
\usepackage{pstricks}
\usepackage[all]{xy}

\usepackage{amssymb,amsmath}
\usepackage[all]{xy}
\usepackage{graphicx}
\usepackage{amscd}
\usepackage{graphicx}
\usepackage{pb-diagram}
\usepackage{pstricks}
\usepackage[all]{xy}
\numberwithin{equation}{section}

\newtheorem{theorem}{Theorem}[section]
\newtheorem{definition}[theorem]{Definition}
\newtheorem{corollary}[theorem]{Corollary}
\newtheorem{example}[theorem]{Example}
\newtheorem{lemma}[theorem]{Lemma}
\newtheorem{assumption}[theorem]{Assumption}
\newtheorem{remark}[theorem]{Remark}
\newtheorem{prop}[theorem]{Proposition}

\newcommand{\RR}{\mathbb{R}}
\newcommand{\QQ}{\mathbb{Q}}

\newcommand{\ZZ}{\mathbb{Z}}

\newcommand{\CP}{\mathcal{P}}

\newcommand{\CM}{\mathcal{M}}

\newcommand{\bX}{\boldsymbol{X}}

\newcommand{\GG}{\Gamma}

\newcommand{\WT}[1]{\widetilde{#1}}
\newcommand{\OL}[1]{\overline{#1}}
\newcommand{\CR}{\textit{crit}}
\begin{document}
\author[Cho]{Cheol-Hyun Cho}
\address{Department of Mathematical Sciences, Research institute of Mathematics\\ Seoul National University\\ San 56-1, 
Shinrimdong\\ Gwanakgu \\Seoul 47907\\ Korea}
\email{chocheol@snu.ac.kr}
\author[Hong]{Hansol Hong}
\address{Department of Mathematical Sciences \\ Seoul National University\\ San 56-1, 
Shinrimdong\\ Gwanakgu \\Seoul 47907\\ Korea}
\email{hansol84@snu.ac.kr}

\title{Orbifold Morse-Smale-Witten complexes}
\thanks{This work was supported by the National Research Foundation of Korea(NRF) grant funded by the Korea government(MEST)(No. 2012-0000795)}
\begin{abstract}
Given a Morse-Smale function on  an effective orientable orbifold, we construct its Morse-Smale-Witten complex. 
We show that critical points of a certain type have to be discarded to build a complex properly, and that gradient flows should be counted with suitable weights. Its homology is proven to be isomorphic to the singular homology of the quotient space under the self-indexing assumption. For a global quotient orbifold $[M/G]$, such a complex can be understood as the $G$-invariant part of the Morse complex of $M$, where the $G$-action on generators of the Morse complex has to be defined including orientation spaces of unstable manifolds at the critical points.

Alternatively in the case of global quotients, we introduce the notion of weak group actions on Morse-Smale-Witten complexes for non-invariant Morse-Smale functions on $M$, which give rise to genuine group actions on the level of homology.
\end{abstract}
\maketitle

\section{Introduction}
Morse theory is one of the most important tools to understand the topology of manifolds. A modern approach to Morse theory, the Morse-Smale-Witten complex, has been exceedingly popular as its infinite dimensional analogue, Floer homology theory, has proven to be a very powerful tool in the area of symplectic and differential geometry. The Morse-Smale-Witten complex (Morse complex for short from now on) is generated by critical points of a Morse function, and the differential on this complex is given by counting the (signed) number of gradient flow lines between critical points of index difference one.

The notion of effective differentiable orbifolds was introduced by Satake \cite{Sa}
under the name ``$V$-manifolds", which is a natural generalization of the notion of differentiable manifolds. Locally it is a quotient space of an Euclidean space by the
effective action of a finite group. 

A Morse function on orbifolds is defined by an invariant function whose local lifts are Morse. In this paper, we define the Morse complex for an effective orbifold with a Morse-Smale function on it. Though Morse inequalities or informations about local Morse data for orbifolds have been known for a while since the work of Lerman-Tolman \cite{LT}, the construction of Morse complex for orbifolds has not been available.

We give a construction of such a complex, and show that its homology is isomorphic to the singular homology of the quotient space (under certain assumptions). There are a few interesting features for orbifold Morse theory.
The first one is that a broken trajectory (with only one breaking) can be a limit of several distinct families of smooth trajectories, whereas it is the limit of a unique family in the case of manifolds. This is due to the fact that there can be several different lifts of broken trajectories, which are not equivalent via local group actions (see Section \ref{analpropMf}).

Secondly, one should count the number of flow lines with suitable weights so that the differential of the chain complex involves data coming from the order of related isotropy groups. Namely, a gradient trajectory should be counted with weights which is determined by the stabilizer of its end point and the one of the trajectory itself.

The most interesting property is doubtlessly that one has to discard critical points of a certain type (called {\em non-orientable critical points}) and consider the subcomplex generated by only orientable critical points to define Morse complex correctly. This is related to the observation already made in \cite{LT} which says that if the local group action does not preserve the orientation of the unstable directions, the topology of the sublevel set of the quotient space does not change when passing through such a critical point. 

This phenomenon has a nice interpretation in the case of global quotient orbifolds.
Recall that the de Rham complex of the global quotient orbifold $[M/G]$ is simply given by the $G$-invariant part of
the de Rham complex of $M$. If we define $G$-action on the Morse complex of $M$ naively by sending
critical point (generator) $x$ to another critical point $gx$, then it turns out that the corresponding $G$-invariant chain subcomplex
does not define the homology of $H^*(M/G)$ (see Example \ref{broken heart}).
Instead, we should really regard each generator $x$ of the Morse complex as an element of the orientation spaces associate to the critical point $x$, which is $\wedge^{top}T_xW^u(x)$. Accordingly, there is an additional
contribution of the $G$-action on orientation spaces (not just points themselves), and the resulting $G$-action
on the new chain complex computes the
singular homology of the quotient space. A critical point $x$ is called {\em orientable} if the $G_x$-action preserves the sign
of $\wedge^{top}T_xW^u(x)$, and {\em non-orientable} otherwise. Non-orientable ones disappear when we take the
$G$-invariant subcomplex automatically.

One drawback is that it is difficult in general to find a function which is both G-invariant and 
Morse-Smale (which is a transversality condition between unstable and stable manifolds of critical points).
For example, Morse functions on orbifolds are dense among smooth functions (by \cite{Wa}, \cite{H}) but they may not be Morse-Smale. This issue is called the equivariant transversality problem. To avoid the usage of abstract perturbation theory,
we provide an alternative approach of weak group actions for global quotients, but we need to assume that a Morse-Smale function exists for general effective orbifolds. That is, we choose any Morse-Smale function on $M$ which is {\em not} necessarily $G$-invariant. And we show that the Morse complex (generated by orientation spaces) of $M$ admits a weak group action of $G$. This action induces a genuine $G$-action on the homology of the Morse complex.

The idea developed in this paper is expected to play a fundamental role to build up orbifold  Lagrangian Floer theory and orbifold Fukaya category associated to it in the sense that the Lagrangian Floer complex for a Lagrangian submanifold in a symplectic manifold generalizes the Morse complex by adding quantum corrections via counting of $J$-holomorphic curves. 
Recall that the work of Dixon, Harvey, Vafa and Witten on the string theory of orbifolds \cite{DHVW}, the discovery of a new ring structure on the cohomology of inertia orbifolds by Chen and Ruan\cite{CR} and orbifold Gromov-Witten invariants \cite{CR2}, have prompted many exciting  developments on the study of orbifolds in the last decades. But Lagrangian Floer theory for orbifolds has not been studied rigorously yet. This paper lays a foundation to define orbifold Fukaya category theory. We expect that new phenomena of orbifold Morse theory which are revealed in this paper should also be presented in orbifold Fukaya category theory, and we will give such a generalization to Novikov or Novikov-Floer theory for
global quotients in a near future.  We remark that for toric orbifolds, Lagrangian Floer theory for  smooth Lagrangian torus fibers has been developed in \cite{CP}.

We also remark that the equivariant cohomology version of Morse-Bott theory has appeared in the work of Austin and Braam \cite{AB} where they dealt with compact connected Lie group $G$. However, their construction does not immediately generalize for a finite group $G$
(which is not connected) due to the orientation space issues. Note that the action of a connected Lie group should be always orientation preserving on orientation spaces. Although it is possible to  present our orbifold as a quotient of a smooth manifold (the frame bundle)
by $SO(n)$, it is difficult to find a connection between their construction and ours as they use the  Cartan model.

Here is the outline of the paper. In Section \ref{MSW_GQ}, Morse theory of global quotient orbifolds is explained with a careful examination of orientation issues. In Section \ref{sec:intrin}, we reformulate the construction made in Section \ref{MSW_GQ} in a more intrinsic form. Analytic properties of obifold Morse-Smale functions are studied in Section \ref{analpropMf}.
In Section \ref{sec:MSWcpxgen}, we define Morse complexes for effective orbifolds and show that $\partial ^2 =0$. In Section \ref{sec:comparison}, we compare the homology of the Morse complex with the singular homology of the quotient space.  In Section \ref{equiv_trans}, we introduce a notion of weak group actions on the Morse complex of a Morse-Smale function $f$ which is not $G$-invariant.
\section*{Acknowledgment}
We thank an anonymous referee for suggestions to improve the exposition of the paper.

\section{Morse-Smale-Witten complexes for global quotients}\label{MSW_GQ}
Let $M$ be a closed oriented connected manifold and suppose that a finite group $G$ acts on $M$ effectively in an orientation preserving way. We denote by $\bX$ the global quotient orbifold $[M/G]$ and denote by $X$ its quotient space and by $\pi : M \to X$ the natural projection map. We fix a $G$-invariant metric on $M$. 
Recall that a smooth function on an orbifold by definition has smooth invariant lift
 on each uniformizing chart. Analogously, a Morse function on an orbifold is defined as follows.
\begin{definition}\label{orbM}
A smooth function $\bar{f}:\bX \to \RR$ is called Morse if every point $\bar{x}$ in the orbifold has a uniformizing chart $(\WT{U}_{ \bar{x} } , G_{\bar{x}}, \pi_{\bar{x}})$ such that $\bar{f} \circ \pi_{\bar{x}}$ is Morse on $\WT{U}_{\bar{x}}$.
\end{definition}
In particular, the $G$-invariant Morse function $f:M \to \RR$ taken above induces a Morse function on the global quotient orbifold $\bX = [M/G]$.

Consider a $G$-invariant Morse function $f : M \to \RR$ which
always exist by \cite{Wa} (see Section \ref{equiv_trans} also).
We write $\bar{f}:\bX \to \RR$ for a function on orbifold $\bX$ or the quotient space $X$.
 \begin{assumption}\label{AssMorseSmale_GQ}
We assume that $f : M \to \RR$ is a $G$-invariant Morse-Smale function.
\end{assumption}
The condition for a gradient vector field $\nabla f$  to satisfy the Morse-Smale transversality condition together keeping
$G$-invariance of $f$ is
not generic, and hence the above is rather a restrictive assumption. For example, we can perturb $\bar{f}:M \to \RR$ to make it Morse-Smale, but lose the $G$-invariance. We postpone the discussion on this issue to the
last section, where we alternatively define weak group actions.

In this section, we construct a Morse complex for a $G$-invariant Morse-Smale function $f$.
We will provide two approaches, the first one is by defining the type (orientable or non-orientable) of a critical point, and considering a $G$-action on the subcomplex generated by orientable critical points (see Subsection \ref{OriCrit}). The second one is
to consider a slightly different chain complex, which is generated by orientation spaces of critical points
whose $G$-invariant part becomes the chain complex constructed in the first approach (Subsection \ref{OriSp}). The second
approach is more natural, but does not generalize to the case of general orbifolds.

\subsection{Orientabilities of critical points }\label{OriCrit}
Lerman and Tolman studied Morse homology of orbifolds in \cite{LT}, where they analyzed the local Morse data near a critical point in orbifolds and proved Morse inequalities for orbifolds. It is already observed in their work that there
are two types of critical points, and intuitively this is mainly because of the Corollary \ref{disk}.
We will see that  the orientation issues are very important even to set up the Morse complex for orbifolds. Even though the group action is assumed to preserve the given orientation of a manifold, it may not preserve the orientations of unstable directions at critical points.


Denote sets of critical points of $f$ and $\bar{f}$ by $\CR(f)$ and $\CR(\bar{f})$, respectively. i.e. $\bar{p} \in \CR(\bar{f})$ if there exists $p \in \CR(f)$ such that $\pi (p) = \bar{p}$. As usual, we define $CM_\ast (M,f)$ as a complex of $\RR$-vector spaces freely generated by $p \in \CR(f)$. Denote by $W^+ (p)$ and $W^- (p)$ the stable and the unstable manifold at $p$ respectively (see for example \cite{Ni}). Also, denote the set of all critical points of $f$ of index $i$ by $\CR_i (f)$. We will write $\mu (\bar{p})$ and $\mu (p)$ for their Morse indices.

A Morse complex of $\bar{f}$ is a certain subcomplex of 
 $C M_\ast (M,f)$ defined in the following way. We first orient $W^{-}(p)$ for each $p \in Crit(f)$ and define the type of a critical point $\bar{p} \in \CR(\bar{f})$.
  
 \begin{definition}\label{def:type}
 If $G_{p}$-action on the unstable manifold $W^-(p)$ at $p \in \pi^{-1} (\bar{p})$ is orientation preserving, then $p$ is called {\em orientable} critical point, and non-orientable otherwise.
Denote by   $\CR^+ (\bar{f})$ (resp. $\CR^- (\bar{f})$ ) the set of all orientable (resp. non-orientable) critical points of $\bar{f}$.

We use the similar notation for critical points of $f$ .
\end{definition}
\begin{remark}
Such a dichotomy was considered already in \cite{LT} and in several subsequent works such as \cite{H}. As observed in \cite{LT}, this is  very natural in terms of local Morse data.  Indeed, we will see later that attaching cells which arise at non-orientable critical points do not contain any topological information for the  quotient space (see Corollary \ref{disk}).
\end{remark}
\begin{remark}
If $G_{p}$ is orientation preserving for one of $p \in \pi^{-1} (\bar{p})$, then
 so is $G_{p'}$ for other $p' \in  \pi^{-1} (\bar{p})$.
\end{remark}

 Let $\CR_i^\pm ({f}) := \CR_i ({f}) \cap \CR^\pm ({f})$. This induces the decomposition 
 $$CM_i (M,f) = C M_i^+ (M,f) \oplus C M_i^- (M,f).$$ 
It is easy to see that $G$-action preserves $CM_i^{+} (M,f)$ and we define 
$$ C M^+_i (X,\bar{f}) :=  CM^+_i (M,f)^G.$$
For $\bar{p} \in \CR(\bar{f})$, we formally write 
 \begin{equation}\label{barp}
[\bar{p}] := \displaystyle\sum_{ p \in \pi^{-1} (\bar{p}) } p,
\end{equation}
and then, $CM^+ (X, \bar{f})(= \displaystyle\oplus_i C M _i^+ (X, \bar{f}) )$ is freely generated by $[\bar{p}]$'s for $\bar{p} \in \CR^+(\bar{f})$. 
 
Next, we define a boundary map $\partial_i : C M_i^+ (X,\bar{f}) \to C M_{i-1}^+ (X,\bar{f}) $. For each orientable critical point $\bar{p}$ of $\bar{f}$, take $G$-invariant orientations on $\{ W^- (p) \, | \, p \in \pi^{-1} (\bar{p}) \}$. For a non-orientable $\bar{p}$, we take an arbitrary orientation on $W^- (p)$. Since $f$ is a Morse function, we have the Morse differential $\partial : CM_i (M,f) \to CM_{i-1} (M,f)$ which is defined as follows:
 
\begin{definition}
For $p,q \in \CR(f)$, define $\WT{\mathcal{M}} (p,q)$ to be the set of all negative gradient flow lines from $p$ to $q$, and by taking quotient under time translation, 
$$\mathcal{M} (p,q) := \WT{\mathcal{M}} (p,q)/\RR.$$  
Then, we define
\begin{equation}\label{eq:bdy}
 \partial p := \sum_{q, \mu(q) = \mu(p)-1} \# \mathcal{M} (p,q) \,\, q.
 \end{equation}
Here, $\# \mathcal{M} (p,q)$ is the number of gradient flow lines in $\mathcal{M} (p,q)$ counted with signs. (See below for the precise sign rule.)
\end{definition}

Now,  we  define a differential $\partial$ for $\bar{f}:\bX \to \RR$ on $[\bar{p}]$ from the formula \ref{barp} and the differential for  $f: M \to \RR$. Namely, we set $\partial [\bar{p}] = \sum_{p \in \pi^{-1} (\bar{p})} \partial p$. We claim that this defines a differential for $C M^+ (X,\bar{f})$. To show this, we need the following two crucial lemmata:

\begin{lemma}\label{cancel}
If $\bar{p} \in \CR^+ (\bar{f})$, then
\begin{equation}\label{bdy1} 
\partial [\bar{p}] \in C M^+(X,\bar{f}).
\end{equation}
i.e. $\partial [\bar{p}]$ has nonzero coefficients only at orientable critical points of $f$.
\end{lemma}

\begin{proof}
Suppose $\bar{p}$ is of index $i$. Then, each $p$ with $p \in \pi^{-1} (\bar{p})$ has index $i$. We will show that the coefficient in $\partial [\bar{p}]$ at an arbitrary non-orientable critical point $q \in \CR_{i-1}^- (f)$ is zero. We set $\mathcal{M} (\bar{p}, q) = \cup_{p \in \pi^{-1} (\bar{p})} \mathcal{M} (p,q)$.


We briefly recall the sign rule for $\mathcal{M} (p,q)$ for reader's convenience.
We fix an orientation of $M$ and that of each unstable manifold. This will orient every stable manifolds so that for each critical point $p$ of $f$, we have
$[W^{+}(p)]_p \, [W^{-}(p)]_p = [M]_p$, where $[\,\,]_p$ means the oriented frames of the
tangent spaces at $p$. Thus, the intersection $W^- (p) \cap W^+(q)$ admits an induced orientation (we follow the orientation conventions in \cite{GP}).

Let $\gamma$ be a negative gradient flow line connecting $p$ and $q$ of relative index $1$. i.e. $\mu (q) =\mu (p) - 1$. Then, ${\rm{Im}} \, \gamma \subset W^- (p) \cap W^+(q)$. If the negative gradient flow orientation of $\gamma$
matches the induced orientation, then it is counted as $+1$, and otherwise as $-1$.

The following convention also produces the same sign rule. Fix a regular value $s \in \RR$ with $f(q) < s < f(p)$. We orient the set $f^{-1}(s)$ so that $[\nabla f][f^{-1}(s)]=[M]$. 
Consider  $S^-(p) := W^- (p) \cap f^{-1}(s) $ which is oriented as a boundary of 
$D^- (p) := W^-(p) \cap f^{-1}([s,\infty))$. Similarly, consider $S^+(q)$
which is oriented as a boundary of $D^+ (q)$. In fact, they are diffeomorphic to $S^{i-1}$ and $S^{n-i}$, respectively. Because $ S^- (p)$ and $ S^+(q) $ are of complementary dimensions in $f^{-1} (s)$, we can count their signed intersection number. One can check that this sign agrees with the former one (following sign rules of \cite{GP}).

Now we prove the lemma.
Suppose that $q$ is a non-orientable critical point with $\mu(q) = \mu(p) -1$.
We split $\mathcal{M} (\bar{p}, q)$ into a disjoint union $\mathcal{M} (\bar{p},q)= \mathcal{M} (\bar{p},q)^+ \sqcup \mathcal{M} (\bar{p},q)^-$ with respect to their signs. Clearly, the sum of these signs will be the coefficient of $q$ in \eqref{eq:bdy}, and hence we have to show that $|\mathcal{M} (\bar{p}, q)^+| = |\mathcal{M} (\bar{p}, q)^-|$. 

Pick any $g \in G_{q}$ which reverses the orientation of $W^- (q)$. Then, $g$ will give a permutation of $\mathcal{M} (\bar{p}, q)$ since $g$ preserves $\pi^{-1} (\bar{p})$. We claim that $g$ sends $\mathcal{M} (\bar{p}, q)^+$ to $\mathcal{M} (\bar{p}, q)^-$. To see this, we consider the action of $g$ on $f^{-1} (s)$. If $g  \cdot p = p'$, then $g$ sends $S^- (p)$ to $S^- (p')$ and preserves $S^+ (q)$. Consider $x \in S^- (p) \cap S^+(q) $ which corresponds to $\gamma \in \mathcal{M} (\bar{p}, q)^+$. By the sign rule given above, $S^- (p)$ and $S^+ (q)$ intersect positively at $x$ in $f^{-1}(s)$, meaning that
$$[ S^- (p)]_x[ S^+ (q) ]_x=[f^{-1}(s)]_x.$$
Since $\bar{p}$ is orientable, our choice of the $G$-invariant orientations on $W^- (p)$'s implies that
$$g \cdot [S^- (p)]_x =[S^- (p') ]_{gx},$$
and since $g$ reverses the orientation of the unstable manifold at $q$,
$$g \cdot [S^+ (q)]_x = - [S^+ (q)]_{gx}.$$
As $g$ preserves the orientation of $M$ and $f$ is $g$-invariant, $g$ preserves the orientation of $f^{-1} (s)$, or $g \cdot [f^{-1}(s)]_x = [f^{-1}(s)]_{gx}$.
Consequently, by considering the oriented frames at $g\cdot x$, we have
$$[ S^- (p')]_{gx}[S^+ (q)]_{gx}
= (g \cdot [S^- (p)]_x)(-g \cdot [S^+ (q)]_x)
= - g \cdot [ f^{-1} (s) ]_x = - [ f^{-1} (s) ]_{gx}.$$
This means $ S^- (p')$ and $ S^+(q)$ intersect negatively at $g \cdot x$. Therefore, the sign of the trajectory $g \cdot \gamma$ is negative. This proves the claim. By the same argument $g^{-1}$ sends $\mathcal{M} (\bar{p},q)^-$ to $\mathcal{M}  (\bar{p}, q)^+$. But since $g$ and $g^{-1}$ are inverses to each other, we get a bijection $g$ from $\mathcal{M} (\bar{p}, q )^+$ to $\mathcal{M} (\bar{p},q)^-$. In particular $|\mathcal{M} (\bar{p}, q)^+| = |\mathcal{M} (\bar{p}, q)^-|$.
\end{proof}

\begin{lemma} The expression $\partial \left( \sum_{p \in \pi^{-1} (\bar{p})} p \right)$ in \eqref{bdy1} is $G$-invariant if $\bar{p}$ is orientable.
\end{lemma}

\begin{proof}
By the previous lemma, \eqref{bdy1} only consists of orientable critical points. Consider two orientable points $q$ and $q' := g \cdot q$ appearing non-trivially in \eqref{bdy1}. We need to show that coefficients of $q$ and $q'$ are equal. However, this is obvious since $g$ and $g^{-1}$ give the sign preserving isomorphisms between $\mathcal{M} (\bar{p}, q) $ and $\mathcal{M} (\bar{p}, q')$. This is because we chose the orientation on the unstable manifold at each orientable critical point of $f$ in a $G$-invariant way.
\end{proof}

We have shown that the Morse boundary map $\partial$ preserves $CM^+_{\ast} (M,f)^G \subset CM_\ast (M,f)$. Thus, we conclude that $ C M^+_\ast (X,\bar{f}) = CM^+_{\ast} (M,\bar{f})^G$ is a subcomplex of $ CM_\ast (M,f)$. We write $CM_\ast (\bX, \bar{f})$ for $ C M^+_\ast (X,\bar{f})$, and use the same notation $\partial$ for the restriction of $\partial : CM_\ast (M,f) \to CM_\ast (M,f)$ to $CM_\ast (\bX,\bar{f})$. Note that $\partial^2 =0$ automatically follows from the property of the Morse boundary of $(M,f)$.

In fact, the resulting homology group $HM_\ast (\bX, \bar{f})$ is isomorphic to the singular homology of the quotient space $X$. We postpone its proof to proposition \ref{comp} in Section 4,   where we deal with it in more general cases.
\begin{theorem}
$HM_\ast (\bX, \bar{f} ) \cong H_\ast (M/G) = H_\ast (X)$.
\end{theorem}

\begin{example}\label{broken heart}
Consider the famous heart $S^2$ with the Morse function $h$ given by the height function $h$ with two maximum $p, q$ , one minimum $s$ and one saddle point $r$.
The heart $S^2$ admits the $\ZZ/2\ZZ$-action generated by the $180$-degree rotation which interchanges $p$ and $q$ and fixes $r$ and $s$. Note that this $\ZZ/2\ZZ$-action preserves $h$.
The quotient space $S^2/(\ZZ/2\ZZ)$ is again $S^2$ topologically. (See Figure \ref{hannah}.)
A chain complex obtained by the naive $G$-action on critical points (without considering orientability) is
$$0  \;\to \;<(p+q)>\; \to\; <r> \;\to\; <s>\; \to \;0.$$
where $<(p+q)>$ denotes the one dimensional vector space generated by $p+q$.
However, the differential here is not squared to be zero, and hence the homology is not well-defined.

\begin{figure}[h]
\begin{center}
\includegraphics[height=1.5in]{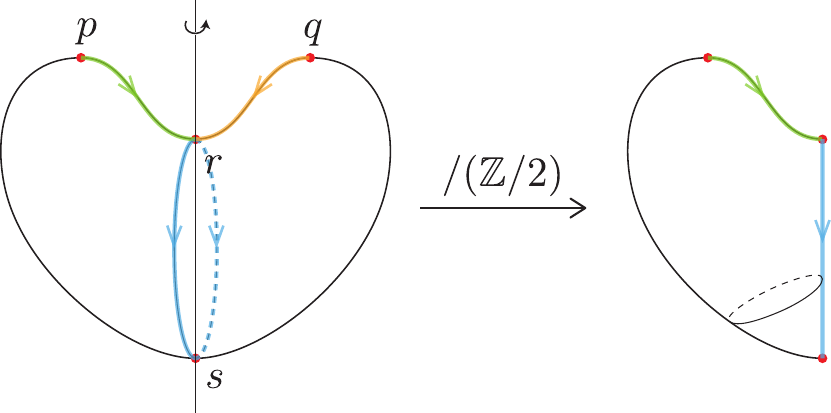}
\caption{The heart $S^2$ equipped with a $\ZZ / 2 \ZZ$-action}\label{hannah}
\end{center}
\end{figure}

The correct  Morse chain complex for the quotient orbifold is
$$0 \;\to\; <(p+q)> \;\to\; 0 \;\to <s>\; \to \;0.$$
Here, critical points $p, q, s$ are orientable whereas the critical point $r$ is non-orientable as the half-rotation reverses the orientation
of the unstable manifold at $r$. Hence, we discard $<r>$ and do not use it as a generator in the above complex.
In this way, we obtain a Morse-homology of $S^2/(\ZZ/2\ZZ)$ which is isomorphic to
the singular homology of $S^2$.


Note that the method of taking invariant subcomplexes obviously does not work for general orbifolds which are not global quotients.
\end{example}

\subsection{Orientation spaces}\label{OriSp}
Now, we explain  an alternative, but more natural definition of the Morse complex for a global quotient orbifold.
Previously, critical points generated Morse complex, but we will consider what is called the orientation space for
each critical point, and consider a Morse complex generated by them. 


We define the orientation space $\Theta_p^-$ at a critical point $p$ of $f$ by $$\Lambda^{\mu (p) }T_p W^- (p),$$
which is isomorphic to $\RR$. (We may define $\Theta_p^- = H_{\mu (p)-1} (W^-(p) \setminus p)$ equivalently.) 
The geometric reason to consider such an orientation space lies in the fact that critical point $p$ essentially represents the cell coming from the unstable 
manifold $W^-(p)$, and hence the group action $p \to g \cdot p$ should really be considered as a map from $W^-(p)$
to $W^-(g \cdot p)$. The last map may or may not be orientation preserving. Therefore, by introducing  $\Theta_p^-$,
we can define a $G$-action on the Morse complex including such orientation data.

We define $CM_\ast (M,f;\Theta)$ by
\begin{equation}\label{CMTheta}
CM_\ast (M,f;\Theta):=\bigoplus_{p \in \CR(f)} \Theta_p^-.
\end{equation}
The Morse boundary operator for \eqref{CMTheta} is defined as follows. We first fix an orientation on $\Theta_p^-$ for each $p \in \CR(f)$. Then, we get a preferred trivialization $\Theta_p^- \stackrel{\cong}{\longrightarrow} \RR \left<p\right>$ which sends the unit vector in the positive direction to $1(=1 \cdot p) \in \RR \left< p \right>$. Therefore, the choice of an orientation on each orientation space gives rise to an isomorphism 
\begin{equation}\label{isomTheta}
CM_\ast (M,f;\Theta) \to CM_\ast (M,f).
\end{equation}
We, then, pull back the usual boundary operator on $CM_\ast (M,f)$ to $CM_\ast (M,f;\Theta)$ by this isomorphism.


$CM_\ast (M,f;\Theta)$ admits a natural $G$-action since $g : W^- (p) \to W^- (g \cdot p)$ induces a $\RR$-linear isomorphism $g_\ast :\Theta_p^- \to \Theta_{g \cdot p}^-$. i.e. $g \in G$ acts on $(p, o)$ ($o \in \Theta_p^-$) by $g \cdot (p,o) = (g \cdot p, g_\ast o)$. By restricting the metric on the ambient space, we have a $G$-invariant metric on each $\Theta_p^-$ so that $||o|| = ||g_\ast o||$. In particular, $g \cdot (p,o) = (p, \pm o)$ for $g \in G_p$ and hence, $G_{(p,o)} =G_p$ if and only if $p$ is an orientable critical point.

By the isomorphism \eqref{isomTheta}, we get another $G$-action on $CM_\ast (M,f)$. Write this  action by $p \mapsto g(p)$. If the induced map $g_\ast : \Theta_p^- \to \Theta_{g \cdot p}^-$ is orientation preserving with respect to the above choice of orientations on $\Theta_{p}^-$ and $\Theta_{g \cdot p}^-$, then $g (p) = g \cdot p$  and, $g(p) = - g \cdot p$ otherwise. Note that the new action of $g$ coincides with the original one for orientable critical points.

\begin{lemma}
The chain complex $CM_\ast (\bX, \bar{f}) \left(=C M^+_\ast (M,f)^G \right)$ is the same as the $G$-invariant subcomplex $CM_\ast (M,f; \Theta)^G$.
\end{lemma}
\begin{proof}
We have seen in the above that two chain complexes agree on the components generated by orientable critical points. It is easy to see from Lemma \ref{cancel} that if $p \in \CR(f)$ is non-orientable, then the component 
$$\bigoplus_{g \cdot p \in \pi^{-1}(\bar{p})} \Theta_{g \cdot p}^-$$
in $CM_\ast (M,f; \Theta)$ is cancelled out after taking $G$-invariants.
\end{proof}

\section{Intrinsic formulae}\label{sec:intrin}
In this section, we will find more intrinsic form of the formula of $\partial$ for the global quotient orbifold $\bX$ in order to extend it to the case of general orbifolds. Consider two orientable critical  points $\bar{p}$ and $\bar{q}$ of $\bar{f}$ of indices $k$ and $k-1$, respectively and suppose that there exists a negative gradient flow line $\bar{\gamma}$ from $\bar{p}$ to $\bar{q}$ in $X$. We want to find the contribution of $\bar{\gamma}$ to the coefficient at $[\bar{q}]$ in $\partial [\bar{p}]$. Recall that $[\bar{p}] = \sum_{ p \in \pi^{-1} (\bar{p}) } p$. Let $\gamma$ be one of liftings of $\bar{\gamma}$.

\begin{lemma}\label{flow}
Given a negative gradient flow line $\gamma$ in $M$, the isotropy groups $G_{x}$ for points $x \in {\rm{Im}} \, \gamma$ are the same.
\end{lemma}

\begin{proof}
This is mainly because the diffeomorphism $\Phi_t$ of $M$ induced by the negative gradient vector field of $f$ is $G$-equivariant. More precisely, let $y$ be another point in $\gamma$. Since both $x$ and $y$ are not critical points of $f$, there exists $t$ such that $\Phi_t (x) = y$. Then, the $G$-equivariance of $\Phi_t$ implies $\Phi_t (g \cdot x ) = g \cdot \Phi_t (x)=g \cdot y$ for any $g \in G$, and hence $G_x = G_y$.
\end{proof}
Thus, we may denote the common isotropy group by $G_\gamma$. It is natural to define $G_{\bar{\gamma}}$ as the conjugacy class represented by $G_\gamma$ in $G$. Then, $|G_{\bar{\gamma}}|$ is well-defined. Note that the number of lifts of $\bar{\gamma}$ in $M$ is exactly $ |G| / |G_{\bar{\gamma}}| $. So, there exist $\left( \sum_{ \bar{\gamma}: \bar{p} \rightarrow \bar{q}} |G| / |G_{\bar{\gamma}}| \right) $-negative flow lines connecting critical points projecting down to $\bar{p}$ and $\bar{q}$.  We want the coefficient of $[\bar{q}]= \sum_{ q \in \pi^{-1} (\bar{q}) } q$ instead of that of single $q$. So, we  divide  $\sum_{\bar{ \gamma }: \bar{p} \rightarrow \bar{q}} |G| / |G_{\bar{\gamma}}| $ by the number of $q$'s in $\pi^{-1} (\bar{q})$. Note that all the coefficients of $q$'s in the sum are equal because of the symmetry coming from the $G$-action. Therefore,
\begin{eqnarray*}
 \partial [\bar{p}] &=& \sum_{\bar{q} \in \CR_{i-1}^+ (\bar{f})} \sum_{\bar{\gamma} : \bar{p} \rightarrow \bar{q}} \epsilon(\bar{\gamma}) \dfrac{1}{|\pi^{-1} (\bar{q})|} \dfrac{|G|}{|G_{\bar{\gamma}}|} \cdot[\bar{q}]  \\
&=&  \sum_{\bar{q} \in \CR_{i-1}^+ (\bar{f})} \sum_{\bar{\gamma} : \bar{p} \rightarrow \bar{q}}  \epsilon(\bar{\gamma}) \dfrac{|G_{\bar{q}}|}{|G|} \dfrac{|G|}{|G_{\bar{\gamma}}|} \cdot [\bar{q}] \\
&=&  \sum_{\bar{q} \in \CR_{i-1}^+ (\bar{f})} \sum_{\bar{\gamma} : \bar{p} \rightarrow \bar{q}}  \epsilon(\bar{\gamma})  \dfrac{|G_{\bar{q}}|}{|G_{\bar{\gamma}}|} \cdot[\bar{q}].
\end{eqnarray*}
Here, $G_{\bar{q}}$ is the conjugacy class of $G_q$, $q \in \pi^{-1} (\bar{q})$. The sum is taken over all orientable critical points $\bar{q}$ of index $k-1$. Also $ \epsilon(\bar{\gamma}) = \pm 1$ assigned to $\bar{\gamma}$ is determined by the sign convention explained before. From now on, we use $\bar{p}$ itself instead of $[\bar{p}]$ for simplicity. 

We denote
$$n(\bar{p},\bar{q}) :=\sum_{\bar{\gamma} : \bar{p} \rightarrow \bar{q}}  \epsilon(\bar{\gamma})  \frac{|G_{\bar{q} } |}{|G_{\bar{\gamma}}|} ,
\;\; \nu_{\bar{q}} (\bar{\gamma}):= \epsilon(\bar{\gamma})  \frac{|G_{\bar{q}}|}{|G_{\bar{\gamma}}|}. $$
 On a minimal chart around $\bar{q}$, the preimage of $\bar{\gamma}$ is $|\nu_{\bar{q}} (\bar{\gamma})|$ copies of gradient flow lines which is obtained by the $G_{\bar{q}}$-action to a single lifting $\gamma$.
 (By a minimal chart, we mean a chart $(\WT{U}_{\bar{q}}, G_{\bar{q}}, \pi_{\bar{q}})$ in which $\WT{U}_{\bar{q}}$ is a connected open subset of an Euclidean space equipped with a linear $G_{\bar{q}}$-action and we assume that there is  unique lifting $q$ of $\bar{q}$ which is the origin).

 So, $\nu_{\bar{q}} (\bar{\gamma})$ can be regarded as a multiplicity of $\bar{\gamma}$ at $\bar{q}$ and $n(\bar{p},\bar{q})$ can be seen as the number of negative gradient flow lines from $\bar{p}$ to $\bar{q}$ counted with multiplicity or weight. We also denote by $\nu_{\bar{p}} (\bar{\gamma}) = \epsilon(\bar{\gamma}) \frac{|G_{\bar{p}}|}{|G_{\bar{\gamma}}|} $ the number of liftings of $\bar{\gamma}$ in an uniformizing chart around $\bar{p}$ counting with signs.

Hence we obtain:
\begin{eqnarray}\label{fin def}
\partial \bar{p} &=&  \sum_{\bar{q}} \left( \sum_{\bar{\gamma} : \bar{p} \rightarrow \bar{q}}  \epsilon(\bar{\gamma})  \dfrac{|G_{\bar{q}}|}{|G_{\bar{\gamma}}|} \right) \bar{q} \\
&=& \sum_{\bar{q}}  \sum_{\bar{\gamma} : \bar{p} \rightarrow \bar{q}} \nu_{\bar{q}} (\bar{\gamma}) \,\, \bar{q} =  \sum_{\bar{q}} n(\bar{p},\bar{q}) \,\, \bar{q}. \nonumber
\end{eqnarray}

We emphasize that $\nu_{\bar{q}} (\bar{\gamma})$ and $n (\bar{p},\bar{q})$ make sense for an arbitrary orbifold $\bX$ with a  given Morse-Smale function. Namely, coefficients of  \eqref{fin def} are intrinsic, only involving data of the critical points of $\bar{f}$, gradient flow lines in the quotient space and local groups at critical points of $\bar{f}$. Note that if the group action is trivial we get the usual formula of the Morse boundary operator. 
In the next section, we shall define a Morse complex of a general orbifold with help of the above formula.

We would like to introduce an alternative formula of the Morse boundary operator which is also intrinsic in a similar sense. The formula is obtained simply by using
$$ \left<\bar{p} \right> := \frac{|G_{\bar{p}}|}{|G|} \sum_{ p \in \pi^{-1}} (\bar{p}) $$
instead of $[\bar{p}]$. Note that $\left< \bar{p} \right>$ can be seen as the average of $p$'s with respect to the $G$-action since $\frac{|G|}{|G_{\bar{p}}|}$ is the cardinality of the orbit containing $p$. With this slight modification of generators, the boundary operator becomes
\begin{equation}\label{underpar}
\partial \left<\bar{p} \right> =  \sum_{\bar{q} \in \CR_{i-1}^+ (\bar{f})} \sum_{\bar{\gamma} : \bar{p} \rightarrow \bar{q}}  \epsilon(\bar{\gamma})  \dfrac{|G_{\bar{p}}|}{|G_{\bar{\gamma}}|} \cdot \left< \bar{q} \right>,
\end{equation}
which is similar to \eqref{fin def}. We use $\underline{\partial}$ to denote the operator with respect to the above choice of generators $\left< p \right>$. 

The resulting homology group is isomorphic to the original one obviously via
$$\psi : \bar{p} \mapsto |G_{\bar{p}}| \cdot \bar{p},$$
which is a chain map with respect to $(\partial, \underline{\partial})$ (with $\RR$-coefficients).

\section{Some properties of Moduli spaces of gradient flow lines}\label{analpropMf}
From now on, let $\bX$ be a compact  oriented connected $n$-dimensional effective orbifold, which may not be a global quotient orbifold. It is still possible to choose a Morse function $\bar{f}$ on $\bX$ (see  \cite{H}) in the sense of definition \ref{orbM}.
\begin{definition}
A Morse function $\bar{f} : \bX \to \RR$ is called Morse-Smale if for $\bar{p}, \bar{q} \in \CR(\bar{f})$ and for $x \in W^- (\bar{p}) \cap W^+ (\bar{q})$, we have
\begin{equation}\label{Orb_Smalecond}
T_x W^- (\bar{p}) + T_x W^+ (\bar{q}) = T_x \bX.
\end{equation}
\end{definition} 
We will make the following assumption. (See Section \ref{equiv_trans} for more explanations.)
\begin{assumption}\label{AssMorseSmale}
We assume that $\bar{f} : \bX \to \RR$ is Morse-Smale \eqref{Orb_Smalecond}.
\end{assumption}
As before, we denote by $X$ the underlying quotient space of $\bX$ and by  $\bar{f}:X \to \RR$ the function induced from $f$.
We define the orientability of critical points as in Definition \ref{def:type} and Morse indices of critical points are defined by using local uniformizing charts. These two notions are independent of the choice of uniformizing charts. Again, we denote by $\CR_k^+ (\bar{f})$ the set of all orientable critical points of $\bar{f}$ of index $k$.
  
For effective orbifolds, analytic properties of $\bar{f}$ can be studied in the following way, which we learned from E. Lerman. For any orbifold $\bX$, its frame bundle $Fr(\bX)$ is a smooth manifold with a smooth, effective, and almost free $O(n)$-action. Then, $\bX$ is naturally isomorphic to the quotient orbifold $[Fr(\bX)/O(n)]$ (see of \cite[Theorem 1.23]{ALR}).  

We start from a general situation where a manifold $M$ is equipped with an action of compact Lie group $G$ which is smooth, effective and locally free. Denote $[M/G]$ by $\bX $ and let $\bar{f} : \bX\to \RR$ be an orbifold Morse-Smale function. Then, we can lift our Morse function $\bar{f}$ to a function $\tilde{f} : M \to \RR$ simply by setting $\tilde{f}:= \bar{f} \circ \pi$. 
\begin{lemma}
$\tilde{f} : M \to \RR$ obtained above is a Morse-Bott function which satisfies the Morse-Smale transversality condition.
\end{lemma}

\begin{proof}
Consider a point $\bar{p}$ and a uniformizing chart $(U_{\bar{p}}, G_{\bar{p}}, \pi_{\bar{p}})$ on which $\bar{f}$ is lifted as a local Morse function $f : U_{\bar{p}} \to \RR$. Let $O_{\bar{p}}$ denote the orbit corresponding to $\bar{p}$. Then, we can identify $U_{\bar{p}}$ with a normal slice to $O_{\bar{p}}$ at a point $p \in O_{\bar{p}}$ which lies over $\bar{p}$. If $\bar{p}$ is not a critical point of $\bar{f}$, then $d \tilde{f} (p) (v) = d f (p) (v) \neq 0$ for $p \in O_{\bar{p}}$ for some $v \in TM$ normal to $O_{\bar{p}}$ since $f$ is Morse. By $G$-invariance, $\tilde{f}$ is regular at all points in the $G$-orbit containing $p$. So, critical sets of $\tilde{f}$ are precisely $O_{\bar{p}} \cong G / G_{\bar{p}}$ for $\bar{p} \in crit(\bar{f})$ which are  submanifolds of $M$. 

Observe that the Hessian of $\tilde{f}$ in the normal direction to the critical submanifold $O_{\bar{p}}$ precisely equals to the Hessian of $f : U_{\bar{p}} \to \RR$ at $p$ (where we identify $U_{\bar{p}}$ with a normal slice of $O_{\bar{p}}$ at $p \in O_{\bar{p}}$). Since $f$ is Morse on $U_{\bar{p}}$, we conclude that the lift $\tilde{f}:M \to \RR$ is Morse-Bott.

Finally, we check the Morse-Smale condition for $\tilde{f}$. Consider two critical manifolds $O_{\bar{p}}$ and $O_{\bar{q}}$ associated with $\bar{p}$ and $\bar{q}$ in $\CR (\bar{f})$. Write $W^- (O_{\bar{p}})$ and $W^+(O_{\bar{q}})$ for the unstable manifold for $O_{\bar{p}}$ and the stable manifold for $O_{\bar{q}}$, respectively. For $x \in W^- (O_{\bar{p}}) \cap W^+(O_{\bar{q}})$, we have to show that 
\begin{equation}\label{MB_Smale}
T_x W^- (O_{\bar{p}}) + T_x W^+(O_{\bar{q}}) = T_x M
\end{equation}
Let $O_x$ be the $G$-orbit containing $x$. Then, since both $W^- (O_{\bar{p}})$ and $W^+(O_{\bar{q}})$ are closed under the $G$-action, the left hand side of \eqref{MB_Smale} contains $T_x O_x$. Therefore, it suffices to check \eqref{MB_Smale} after restrict it to the normal direction to the orbit $O_x$, which is equivalent to \eqref{Orb_Smalecond}.  
\end{proof}

The moduli space of gradient trajectories for $\tilde{f}$ on $M$ was already studied in depth. For example, the convergence and the gluing of gradient flow lines of Morse-Bott functions are discussed in \cite{AB} and \cite{BH1}. Making use of these nice features of $\tilde{f}$, we show that connected components of the moduli space of gradient trajectories of $\bar{f}$ can be compactified. Here, the moduli space of gradient trajectories is thought of as the quotient by the time translation action of $\RR$.

\begin{prop}\label{orbstrCP}
Consider two critical points $\bar{p}, \bar{q} \in crit^+(\bar{f})$ of index difference 2. Then, any connected component of the moduli space of gradient flow trajectories of $\bar{f} : \bX \to \RR$ from $\bar{p}$ to $\bar{q}$ can be compactified so that it becomes a compact oriented $1$-dimensional orbifold with boundary.
\end{prop}

\begin{proof}
Consider (negative) gradient flow lines from $\bar{p}$ to $\bar{q}$ and denote by $\mathcal{M} (\bar{p}, \bar{q})$ the set of all such flow lines modded out by the time translation. We want to  compactify a connected component $\mathcal{P}$ of $\mathcal{M} (\bar{p}, \bar{q})$ to get an 1-dimensional orbifold (with boundary) $\bar{\mathcal{P}}$. We will describe it as a quotient of a compact manifold with boundary by a locally free group action which preserves the boundary.

Define $\mathcal{M} (O_{\bar{p}}, O_{\bar{q}})$ to be the moduli space of all flow lines from $O_{\bar{p}}(:=\pi^{-1} (\bar{p}) )$ to $O_{\bar{q}}(:=\pi^{-1} (\bar{q}))$. Consider the preimage $\tilde{\mathcal{P}}$ of flow lines from $\bar{p}$ to $\bar{q}$ in the component $\mathcal{P}$ under the quotient map $\pi : M \to \bX$, which is a subset of $\mathcal{M} (O_{\bar{p}}, O_{\bar{q}})$. $\mathcal{M} (O_{\bar{p}}, O_{\bar{q}})$ admits a natural $G$-action as $\tilde{f}$ is $G$-invariant and the time translation action commutes with the $G$-action. Then, $\tilde{\mathcal{P}}$ is given by the union of several connected components of $\mathcal{M} (O_{\bar{p}}, O_{\bar{q}})$ which is closed under the $G$-action, since the $G$-action on $M$ is locally free and $\mathcal{P}$ is a connected component of $\mathcal{M}^{orb} (\bar{p}, \bar{q}):=[\mathcal{M} (O_{\bar{p}}, O_{\bar{q}})/G]$. 

If we denotes one of components of $\tilde{\mathcal{P}}$ by $\tilde{\mathcal{P}}_0$ and the maximal subgroup of $G$ preserving $\tilde{\mathcal{P}}_0$ by $H$, then $\tilde{\mathcal{P}}$ is the union of $[G : H]$-copies of $\tilde{\mathcal{P}}_0$. ($H$ should have a finite index in $G$ because $\tilde{\mathcal{P}}_0$ is a maximal connected subset of $\tilde{\mathcal{P}}$.) Now, $\mathcal{P}$ can be described as a quotient of $\tilde{\mathcal{P}}_0$ by the action of $H$.

We can compute the dimension of $\tilde{\mathcal{P}}_0$ precisely. Note that $\dim \tilde{\mathcal{P}}_0 = \dim \mathcal{M} (O_{\bar{p}}, O_{\bar{q}})$ and $\mathcal{M} (O_{\bar{p}}, O_{\bar{q}}) = W^- (O_{\bar{p}}) \cap W^+ (O_{\bar{q}}) / \RR$, where $W^- (O_{\bar{p}})$ and $W^+ (O_{\bar{q}})$ are unstable manifold for $O_{\bar{p}}$ and $O_{\bar{q}}$, respectively. As the $G$-action on $M$ is locally free,
\begin{equation*}
\begin{array}{rl}
\dim W^- (O_{\bar{p}}) &= \dim G + \mu (\bar{p}),\\
\dim W^+ (O_{\bar{q}})  &= \dim G + \dim W^+ (\bar{q}) \\
&=\dim G + (\dim \bX - \mu (\bar{q})).
\end{array}
\end{equation*}
Therefore,
\begin{eqnarray*}
\dim \mathcal{M} (O_{\bar{p}}, O_{\bar{q}}) &=&  \dim W^- (O_{\bar{p}}) + \dim W^+ (O_{\bar{q}} ) - \dim M -1\\ 
&=& \dim G + 1,
\end{eqnarray*}
and hence $\dim \tilde{\mathcal{P}}_0 = \dim G +1$, also.

The $G$-action on $\mathcal{M} (O_{\bar{p}}, O_{\bar{q}})$ is orientation preserving since the $G$-action on $W^+ (O_{\bar{p}})$ and $W^- (O_{\bar{q}})$ is orientation preserving. ($G$-actions on critical submanifolds obviously preserve orientations and so are they on normal directions because $\bar{p}, \bar{q} \in crit^+ (\bar{f})$.) Therefore, $H$ acts on $\tilde{\mathcal{P}}_0$ preserving its orientation as well.

We compactify $\tilde{\mathcal{P}}_0$ by adding broken trajectories from $\bar{p}$ to $\bar{q}$ as usual. (See \cite{AB}, \cite{BH} or, \cite{BH1}.) Suppose that the broken trajectory $(\gamma_1, \gamma_2)$ lies in the compacitification $\tilde{\mathcal{P}}_0^{cpt}$ of $\tilde{\mathcal{P}}_0$. Then, there exists a family of trajectories $\{\gamma_t \}_t \subset \tilde{\mathcal{P}}_0$ which converges to $ (\gamma_1, \gamma_2)$. For $h \in H$, $\{h \cdot \gamma_t \}_t$ is also contained in $\tilde{\mathcal{P}}_0$ by $H(\leq G)$-invariance of $\tilde{\mathcal{P}}_0$ and $\tilde{f}$. Moreover, they converges to $(h \cdot \gamma_1, h \cdot \gamma_2)$. This shows that the $H$-action on $\tilde{\mathcal{P}}_0$ can be extended naturally to $\tilde{\mathcal{P}}_0^{cpt}$.

In general, $\tilde{\mathcal{P}}_0^{cpt}$ has a structure of a manifold with corner, but in our case there are only codimension $1$ strata since $\mu (\bar{p}) - \mu (\bar{p}) = 2$. Note that $H$ preserves $\tilde{\mathcal{P}}_0^{cpt} \setminus \tilde{\mathcal{P}}_0$ since it sends broken trajectories to broken trajectories. Since, $\mathcal{P} \cong [\tilde{\mathcal{P}}_0 / H] \subset [\tilde{\mathcal{P}}_0^{cpt} / H]$, we can think of $[\tilde{\mathcal{P}}_0^{cpt} / H]$ as a compactificaton of $\mathcal{P}$. Denote $[\tilde{\mathcal{P}}_0^{cpt} / H]$ by $\bar{\mathcal{P}}$. Now, we show that the $H$-action on $\bar{\mathcal{P}}$ is locally free, which will imply that $\bar{\mathcal{P}}$ has a structure of an orbifold with boundary.

We claim that if $h \in H$ fixes $[\gamma] \in \tilde{\mathcal{P}}_0^{cpt}$ (i.e. if $h \cdot \gamma$ is a time translation of $\gamma$), then $h$ fixes every points on $\gamma$. Suppose to the contrary that two different points $x, y \in M$ are both on $\gamma$ and $h$ sends $x$ to $y$. Since $\gamma$ is a negative gradient flow line of $f$, $f(x) \neq f(y)$ which contradicts the fact that $f$ is $H$-invariant. So each $\gamma \in \bar{\mathcal{P}}$ has a finite isotropy group. Therefore, $\bar{\mathcal{P}}$ is an orbifold with boundary whose interior is isomorphic to $\mathcal{P}$. It is 1-dimensional since $\dim G = \dim H$ (i.e. $[G:H]$ is finite) and $\dim \tilde{\mathcal{P}}_0^{cpt} = \dim G +1$.
\end{proof}

\begin{remark}
$\bar{\mathcal{P}}$ is not an effective orbifold in general. See Lemma \ref{inteq} where we describe the oribfold structure of $\bar{\mathcal{P}}$ more precisely.
\end{remark}

Exactly the same argument in the proof can be applied to the case when $\mu (\bar{p}) - \mu (\bar{q}) =1$. (Even easier, since compactification with broken trajectories is not needed.) So, we get the following corollary.

\begin{corollary}
If $\mu (\bar{p}) - \mu (\bar{q}) =1$ for $\bar{p}, \bar{q} \in \CR^+ (\bar{f})$, then there are only finitely many negative gradient flow lines from $\bar{p}$ to $\bar{q}$.
\end{corollary}

Let us consider the $G$-action on the whole compactified moduli space $\overline{\mathcal{M}} (O_{\bar{p}}, O_{\bar{q}})$. It is also a union of several copies which are isomorphic to $\tilde{\mathcal{P}}_0^{cpt}$. However, at the boundary of each connected component, there are additional group action of $G \times G$ where $(g_1, g_2) \cdot ([\gamma_1, \gamma_2]):= ([g_1 \cdot \gamma_1], [g_2 \cdot \gamma_2])$. Note that two broken trajectories $([\gamma_1, \gamma_2])$ and $(g_1, g_2) \cdot ([\gamma_1, \gamma_2])$ have the same image in the quotient space $\bX= [M/G]$. This phenomenon is responsible for the strange shape of $\mathcal{M} (\bar{p}, \bar{q})$ as shown in (a) of Figure \ref{1dimmoduli}. (See Example \ref{starlike}.)

From now on, we omit ``$[ \,\, ]$" for simplicity and write $\gamma$ instead of $[\gamma]$.

 \section{Morse-Smale-Witten complexes of General Orbifolds}\label{sec:MSWcpxgen}

We are now ready to construct Morse complexes of orbifolds which are not necessarily global quotients. Suppose that we have a Morse-Smale function $\bar{f} : \bar{X} \to \RR$ (Assumption \ref{AssMorseSmale}) and let $CM_k (\bX, \bar{f})$ 
be the $\RR$-vector space generated by $\CR_k^+ (\bar{f})$. By using the notation of \eqref{fin def}, we define
\begin{equation}\label{orbMorsebdy}
\partial \bar{p} = \sum_{\bar{q} \in \CR_{k-1}^+ (\bar{f})} n(\bar{p},\bar{q}) \, \bar{q} =  \sum_{\bar{q} \in \CR_{k-1}^+ (\bar{f})} \, \sum_{\bar{\gamma} \in \mathcal{M} (\bar{p},\bar{q})} \nu_{\bar{q}} (\bar{\gamma}) \, \bar{q},
\end{equation}
for $\bar{p} \in \CR_k^+ (\bar{f})$.
Then, the main theorem of this section is 
\begin{theorem}\label{thm:main}
$( CM_\ast (\bX, \bar{f}), \partial)$ defines a chain complex. i.e $\partial^2=0$
\end{theorem}
The proof of this theorem will occupy the rest of the section. 
Before we proceed for its proof, we explain the main difference from the case of smooth manifolds.

\begin{figure}[h]
\begin{center}
\includegraphics[height=2in]{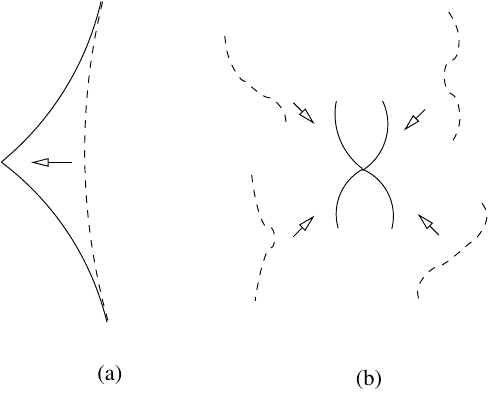}
\caption{limits of gradient flows for (a) manifolds  and (b) orbifolds}\label{pic:flow}
\label{flow1}
\end{center}
\end{figure}

Recall that the standard argument to show $\partial^2=0$ uses the compactifications of moduli spaces of negative gradient flow lines between critical points of index difference $2$. The same analysis yet works as we have seen in Proposition \ref{orbstrCP}, but there is a crucial difference which we shortly mentioned at the end of the previous section. Namely, a given broken trajectory (representing $\partial ^2$) on $X$ can become a limit of several families of trajectories which are distinct even after modding out by the local group action. (See Figure \ref{pic:flow}.) This phenomenon can be examined more clearly in the following example.

\begin{example}\label{starlike}
Let $\bar{\gamma}$ be a negative gradient flow line from $\bar{p}$ to $\bar{q}$ and $\bar{\delta}$ from $\bar{q}$ to $\bar{r}$. Assume for simplicity
\begin{equation}\label{ass:exex}
 G_{\bar{\gamma}} = G_{\bar{\delta}} = 1.
\end{equation}
 On an uniformizing neighborhood $(\WT{U}_{\bar{q}}, G_{\bar{q}}, \pi_{\bar{q}})$, there are $|G_{\bar{q}}|$-flow lines which lift $\bar{\gamma}$ and also $|G_{\bar{q}}|$-flow lines which lift $\bar{\delta}$. 

Choose $\gamma$ and $\delta$ to be one of the flow trajectories covering $\bar{\gamma}$ and $\bar{\delta}$ respectively in the cover $\WT{U}_{\bar{q}}$. Then, $(\gamma,\delta)$ lifts the broken trajectory $(\bar{\gamma},\bar{\delta})$ in $\bX$. For each pair $(g_1,g_2)$ in $G_{\bar{q}}$, $g_1 \cdot \gamma$ together with
$g_2 \cdot \delta$ give another broken trajectory in the cover which projects down to $(\bar{\gamma},\bar{\delta})$.

In this way, we find $|G_{\bar{q}}|^2$-broken trajectories in $\tilde{U}_{\bar{q}}$ lying over $(\bar{\gamma},\bar{\delta})$ and hence, we have $|G_{\bar{q}}|^2$-families of smooth gradient flow lines converging to one of $|G_{\bar{q}}|^2$-broken trajectories in $\WT{U}_{\bar{q}}$. From the assumption \eqref{ass:exex}, the $G_{\bar{q}}$-action on the set of broken trajectories in $\WT{U}_{\bar{q}}$ is free and so is on the set of (local) gluings in $\WT{U}_{\bar{q}}$. Therefore, we get $ |G_{\bar{q}}|(=|G_{\bar{q}}|^2 / |G_{\bar{q}}| )$-families after taking quotient by the $G_{\bar{q}}$-action. We see that there are $|G_{\bar{q}}|$ distinct families of smooth trajectories converging to a single broken trajectory $(\bar{\gamma}, \bar{\delta})$ near $\bar{q}$ in this case.
\end{example}

\begin{figure}[h]
\begin{center}
\includegraphics[height=1.7in]{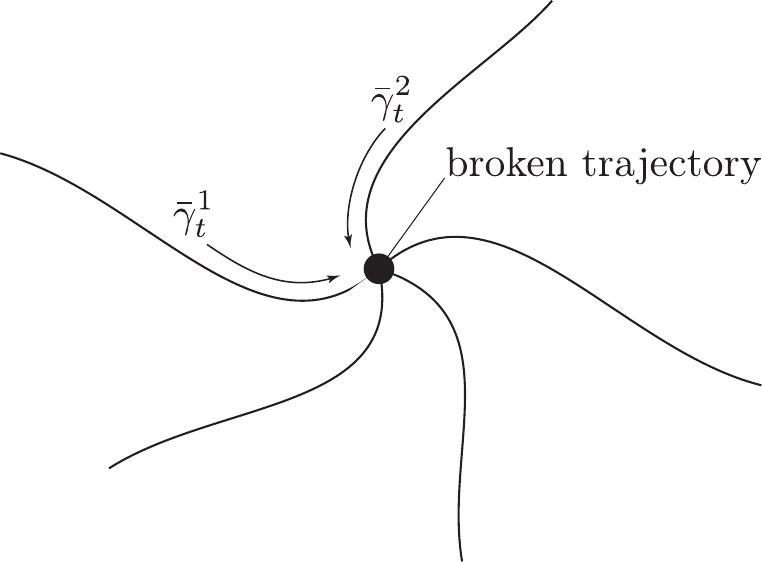}
\caption{Two limiting trajectories in the 1 dimensional moduli space of flow lines}\label{1dimmoduli}
\label{1moduli}
\end{center}
\end{figure}
Consider the moduli space of gradient flow trajectories
between critical points of index difference two, which is of dimension one. The above example
illustrates that near each broken trajectory the compactified moduli space looks like a join of several copies of interval $[0,1)$ at $0$'s ($0$ corresponds to the broken trajectory) equipped with a $G_{\bar{q}}$ action.
Recall that the uncompactified moduli space has a natural orbifold structure (Proposition \ref{orbstrCP}). This orbifold structure also can be understood locally since a local 1-parameter family of trajectories in $\mathcal{P}$ lies in a single uniformizing cover in a way compatible with the local group action. Observing locally around a breaking point of the broken trajectory, the orbifold structure of each limiting trajectory to a given broken trajectory may not be isomorphic to each other in general. In figure \ref{1dimmoduli}, two limiting trajectories $\{ \bar{\gamma}_t^1 \}$ and $\{ \bar{\gamma}_t^2 \}$ have the same broken trajectory as a limit, but the orbifold structures along $\{ \bar{\gamma}_t^1 \}$ and $\{ \bar{\gamma}_t^2 \}$ can be quite different. We will see that the difference of the orbifold structure gives rise to certain weights in the formula of the orbifold Morse boundary operator \eqref{orbMorsebdy}.\\

In order to prove the main theorem, we first prove a couple of lemmata on stabilizers of the gradient flow trajectories.\\

We set the notations as follows. Consider $\bar{p} \in \CR_{k}^+ (\bar{f}),  \bar{q},\bar{q}' \in \CR_{k-1}(\bar{f}),  \bar{r} \in \CR_{k-2}^+(\bar{f})$. Note that $\bar{q}$ and $\bar{q}'$ are not assumed to be orientable. Let $\bar{\gamma}$ (resp. $\bar{\gamma}'$)  be negative gradient flow lines from $\bar{p}$ to $\bar{q}$ (resp. $\bar{q}'$) and let $\bar{\delta}$(resp. $\bar{\delta}'$) be flow lines from $\bar{q}$ (resp. $\bar{q}'$) to $\bar{r}$. Suppose that two broken trajectories $(\bar{\gamma}, \bar{\delta})$ and $(\bar{\gamma}' , \bar{\delta}')$ are connected by 1-parameter family of negative gradient flow lines from $\bar{p}$ to $\bar{r}$. Take the set of flows lines in the above 1-parameter family  and call it $\CP$. 

\begin{remark}
Even if $\CP$ flows between two orientable critical points, a breaking point (either $\bar{q}$ or $\bar{q}'$) of a broken trajectory in the limit is not necessarily orientable. This is the reason why we did not impose any condition on the orientability of $\bar{q}$ and $\bar{q}'$. Indeed, example \ref{broken heart} already shows this phenomenon. We shall show that, however, the broken trajectory itself has a well defined sign as a boundary of $\CP$.
\end{remark}

\begin{lemma}\label{inteq}
$\CP$ is an one dimensional oriented orbifold whose stabilizers $G_{\bar{\gamma}}$ are all isomorphic for each $\bar{\gamma} \in \mathcal{P}$. (Thus, it is an ineffective orbifold for nontrivial $G_{\bar{\gamma}}$.)
\end{lemma}
\begin{proof}
We already know that $\mathcal{P}$ is an oriented 1-dimensional orbifold from Proposition \ref{orbstrCP}. So, we only have to show that stabilizers are all isomorphic. But, this is clear since any connected one dimensional orientable orbifold satisfies such a property. Note that a finite group action on an interval say $(-1,1)$
is either identity or $x \mapsto -x$ up to diffeomorphism. The latter cannot be orientation preserving. Therefore local groups act trivially, and hence the stabilizers are isomorphic to each other.
\end{proof}

Recall from Proposition \ref{orbstrCP} that we have a natural compactification $\OL{\CP}$ of each component $\CP$ (in $\mathcal{M}^{orb} (\bar{p}, \bar{q})$) which is obtained by adding limit broken trajectories $(\bar{\gamma}, \bar{\delta})$, $(\bar{\gamma}', \bar{\delta}')$ to $\CP$.
We now look into the orbifold structure of $\OL{\CP}$ more in detail. In particular, we compare the stabilizers of the limiting trajectories and that of $\CP$.
 
 Consider the uniformizing chart $(\tilde{U}_{\bar{q}}, G_{\bar{q}}, \pi_{\bar{q}})$ around $\bar{q}$ with $U_{\bar{q}} = \pi_{\bar{q}} (\tilde{U}_{\bar{q}} )$. Let $\Gamma$ be the set of all liftings of $\bar{\gamma} \cap U_{\bar{q}}$ and $\Delta$ be that of $\bar{\delta} \cap U_{\bar{q}}$. Then $G_{\bar{q}}$ naturally acts on $\Gamma \times \Delta$ by the diagonal action.
Since there is a unique gluing for a given broken trajectory in the uniformizing cover, the quotient set  $\Gamma \times \Delta/G_{\bar{q}}$ by the diagonal action can be seen as the set of all possible smooth trajectories converging to $(\bar{\gamma}, \bar{\delta})$
 in $\bX$.  Here, we recognize that there can be several different gluings for the single broken trajectory $(\bar{\gamma}, \bar{\delta})$. In summary,
 
\begin{lemma}\label{CP}
 $\CP$ determines an element of $\Gamma \times \Delta/G_{\bar{q}}$, say  $ [ \gamma, \delta  ] \in \Gamma \times \Delta/G_{\bar{q}}$ and this correspondence is one to one locally around $\bar{q}$.
\end{lemma}

Now, consider the isotropy groups $G_{\gamma}$ and $G_\delta$ of $\gamma$ and $\delta$. Their intersection  $G_{\gamma} \cap G_{\delta} \subset G_{\bar{q}}$ is regarded as the isotropy group at the boundary point $(\bar{\gamma}, \bar{\delta})$ of $\CP$. We denote its conjugacy class by $G_{[\gamma,  \delta ]}$. In general, the limit of isotropy groups are always a subgroup of the isotropy group at the limit point. For the moduli space of gradient flow trajectories,  we also have the converse, and it will be crucial in proving $\partial^2=0$.

\begin{lemma}\label{bdy}
$G_{[\gamma, \delta ]} \cong G_{\bar{c}}$  for any $\bar{c} \in \mathcal{P}$.
\end{lemma}
\begin{proof}
We prove that $G_{[\gamma, \delta] } \cong G_{\bar{c}}$, for sufficiently close $\bar{c}$ to the boundary point $(\bar{\gamma}, \bar{\delta})$. This will be enough by lemma \ref{inteq}. Take a uniformizing neighborhood around $\bar{q}$, $(\WT{U}_{\bar{q}}, G_{\bar{q}}, \pi_{\bar{q} })$ and consider the lifting $\WT{\mathcal{P}}$ of one parameter family converging to one of liftings $(\gamma, \delta)$ of $(\bar{\gamma}, \bar{\delta})$. By taking two different slice of $f$ meeting ${\gamma}$ and ${\delta}$ respectively, the usual continuity argument shows that $G_{x} \subset G_{\gamma}$ and $G_{x} \subset G_{\delta}$ for $x$ (in the slice of $f$) with $\pi_{\bar{q}} (x) \in {\rm Im} \, \bar{c}$. Therefore $G_{\bar{c}} \subset G_{[\gamma, \delta]}$.

Conversely, assume there exists $g \in G_{\bar{q}}$ which fixes $(\gamma, \delta)$ but does not fix $\WT{\mathcal{P}}$. Then, $g \cdot \WT{\mathcal{P}}$ would be a different family from $\WT{\mathcal{P}}$ converging to the same limit $(\gamma, \delta )$. This is impossible in the standard Morse theory point of view on the uniformizing chart.
\end{proof}

\begin{remark}
Note that $G_{\gamma}$'s are conjugate to each other for liftings $\gamma$ of $\bar{\gamma}$, but the intersection $G_{\gamma} \cap G_{\delta}$ depends on each choice of lifts and its cardinality may also depend on the choice of lifts.
\end{remark}

Therefore, the set $\overline{\mathcal{P}}$ is an ineffective orbifold, where the same isotopy group acts on every points trivially. Also, it carries a natural orientation. We will prove that the natural orientation at the boundary broken trajectories of $\overline{\mathcal{P}}$ are opposite to each other, using almost the same argument in Morse theory for smooth manifolds. To do this, we first introduce a sign rule for the ``boundary" of $\overline{\CP}$ by showing 
that $ (\Gamma \times \Delta ) / G_{\bar{q}}$ inherits signs from $\Gamma \times \Delta$ naturally. Namely, we have:

\begin{lemma}
For $(\gamma, \delta) \in \Gamma \times \Delta$, define $\epsilon(\gamma, \delta)$ as the product of $\epsilon(\gamma)$ and $\epsilon(\delta)$. Regardless of $\bar{q}$ being orientable or not, the $G_{\bar{q}}$-action on $\Gamma \times \Delta$ preserves $\epsilon (\gamma, \delta)$. i.e. 
$$\epsilon (g \cdot \gamma, g \cdot \delta) = \epsilon(\gamma, \delta)$$
for all $g \in G_{\bar{q}}$.
\end{lemma}

\begin{proof}
If $\bar{q}$ is orientable, we can give $\bar{\gamma}$ and $\bar{\delta}$ well-defined signs. Clearly, $\epsilon[\gamma, \delta] = \epsilon(\bar{\gamma}) \cdot \epsilon(\bar{\delta})$ for all $[\gamma, \delta] \in \Gamma \times \Delta / G_{\bar{q}}$ since $G_{\bar{q}}$-action preserves all signs in concern.

Suppose $g \in G_{\bar{q}}$ reverses the orientation of $W^- (q)$. Since $\bar{p}$ and $\bar{q}$ are both orientable, exactly the same argument in Lemma \ref{cancel} shows that $\epsilon (g \cdot \gamma) = - \epsilon(\gamma)$ and $\epsilon (g \cdot \delta) = - \epsilon(\delta)$. This proves the lemma.
\end{proof}

From the lemma, the following sign rule makes sense.

\begin{definition}
For $[ \gamma, \delta] \in \Gamma \times \Delta / G_{\bar{q}}$, we assign a sign to it by letting
$$\epsilon [\gamma, \delta] := \epsilon(\gamma, \delta) =  \epsilon(\gamma) \cdot \epsilon(\delta).$$
\end{definition}

As a result, the orientation of broken trajectories $(\bar{\gamma}, \bar{\delta})$, as a boundary of any component $\CP$ converging to it, is equally given by $\epsilon(\bar{\gamma}) \cdot \epsilon(\bar{\delta})$. Consequently, the orientation issue of $\overline{\CP}$ can be rephrased as follows:

\begin{lemma}
 If there is one parameter family $\mathcal{P}$ which corresponds to $[\gamma, \delta]$ and $[\gamma', \delta']$ in the sense of lemma \ref{CP}, $\epsilon[\gamma,\delta]$ and $\epsilon[\gamma',\delta']$ should be opposite to each other. 
\end{lemma}

As we mentioned, the proof is not at all different from the classical one. (See \cite{AB} for example). This sign cancellation directly proves $\partial^2 =0$ in smooth case. However, in order to count gradient flow trajectories and to describe cancellation phenomenon in orbifold case correctly, we should take into account the orbifold structure additionally as we take ``weighted" sums for the Morse boundary operator for orbifolds \eqref{orbMorsebdy}.

\begin{definition}
For the compactification  $\overline{\mathcal{P}}$ as above, 
the following expression will be called the weighted boundary of  $\overline{\mathcal{P}}$:
$$\partial \OL{\mathcal{P}} = \dfrac{ \epsilon[\gamma, \delta] }{|G_{[\gamma, \delta]}|} (\bar{\gamma}, \bar{\delta}) + \dfrac{\epsilon[\gamma', \delta']}{|G_{[\gamma', \delta']}|} (\bar{\gamma}', \bar{\delta}').$$
We call the numbers $\frac{ \epsilon[\gamma, \delta] }{|G_{[\gamma, \delta]}|} $, $\frac{\epsilon[\gamma', \delta']}{|G_{[\gamma', \delta']}|}$ the weights  and write them as $\omega_{\mathcal{P}}(\bar{\gamma}, \bar{\delta})$, $\omega_{\mathcal{P}}(\bar{\gamma}', \bar{\delta}')$, respectively.
\end{definition}

Now, the standard arguments of proving $\partial^2=0$ in the smooth case together with the above choice of weights give the following equation. (Lemma \ref{bdy})
\begin{equation}\label{bdysum1}
\sum_{(\bar{\zeta},\bar{\eta}) \in \partial \OL{\mathcal{P}}} \omega_{\mathcal{P}}( \bar{\zeta},\bar{\eta} ) = 0
\end{equation}

Denote by $\OL{\CM}(\bar{p},\bar{r})$  the compactified moduli space of negative gradient flow lines from $\bar{p}$ to $\bar{r}$.
As explained in the beginning of the section, geometrically this is given by several copies of
compact intervals (equipped with trivial actions of corresponding isotropy groups) which
are possibly joined at boundary points if they define families of flow lines sharing the same limit.
Also note that the limiting flows to a fixed broken trajectory might have non-isomorphic stabilizers
by Lemma \ref{bdy}. So we cannot really think of $\OL{\CM}(\bar{p},\bar{r})$ as an orbifold with boundary.
This is somewhat different from the smooth case where compactified moduli spaces are
manifolds with corners.
Denote 
 $\partial \OL{\mathcal{M}} (\bar{p},\bar{r}): = \OL{\mathcal{M}} (\bar{p},\bar{r}) - \CM(\bar{p},\bar{r})$.

\begin{definition}
 If $ (\bar{\zeta},\bar{\eta}) \in \partial \OL{\mathcal{M}}(\bar{p},\bar{r})$,  we define
\begin{equation}\label{weight}
\omega ( \bar{\zeta},\bar{\eta})  := \sum_{\mathcal{P} \, \textrm{with}\, (\bar{\zeta},\bar{\eta}) \in \partial \OL{\mathcal{P}}} \omega_{\mathcal{P}} (\bar{\zeta},\bar{\eta}),
\end{equation}
where the sum is taken over all 1-parameter family $\mathcal{P}$  one of whose boundary is $(\bar{\zeta},\bar{\eta})$. Finally, we denote  the sum of all weights associated to the gluings converging to one of broken trajectories through $\bar{p}$, $\bar{q}$ and $\bar{r}$ as
\begin{equation}\label{brk}
\omega(\bar{p},\bar{q},\bar{r}) := \sum_{(\bar{\zeta},\bar{\eta}) \in \mathcal{M}(\bar{p},\bar{q}) \times \mathcal{M}(\bar{q},\bar{r} )} \omega(\bar{\zeta},\bar{\eta}).
\end{equation}
For $ \partial \OL{\mathcal{M}}(\bar{p},\bar{r})$ (which is the set of all broken trajectories from $\bar{p}$ to $\bar{r}$), note that we have $\mathcal{M}(\bar{p},\bar{q}) \times \mathcal{M}(\bar{q},\bar{r}) \subset \partial \OL{\mathcal{M}} (\bar{p},\bar{r})$ for any $\bar{q}$.
\end{definition}

Since all intervals contained $\OL{\mathcal{M}} (\bar{p},\bar{r})$ are oriented so that they are compatible with their boundary orientations, all terms in the sum of \eqref{weight} have the same signs.

 From \eqref{bdysum1}, we get the following equality.
\begin{equation}\label{bdysum2}
\sum_{\bar{q} \in \CR_{k-1} (\bar{f})} \omega(\bar{p},\bar{q},\bar{r}) =\sum_{(\bar{\zeta},\bar{\eta}) \in \partial \OL{\mathcal{M}} (\bar{p},\bar{r})} \omega(\bar{\zeta},\bar{\eta})= \sum_{\mathcal{P}} \sum_{(\bar{\zeta},\bar{\eta}) \in \partial \OL{\mathcal{P}} } \omega_{\mathcal{P}} (\bar{\zeta},\bar{\eta}) = 0
\end{equation}

Figure \ref{moduli} below explains how we sum up weighted contributions near orientable critical points.
Paths in the figure represent  (oriented) 1-dimensional moduli spaces, and these converge to broken trajectories, which are drawn as $\circ$'s. 
In $(b)$ of Figure \ref{moduli},  $\lambda_1 + \lambda_2+ \lambda_3 +\lambda_4  + \lambda_5$ contributes 
to  $\omega (\bar{p}, \bar{q}, \bar{r})$, and summation on a neighboring dotted circle contributes
to $\omega (\bar{p}, \bar{q}', \bar{r})$.

\begin{figure}[h]
\begin{center}
\includegraphics[height=2.8in]{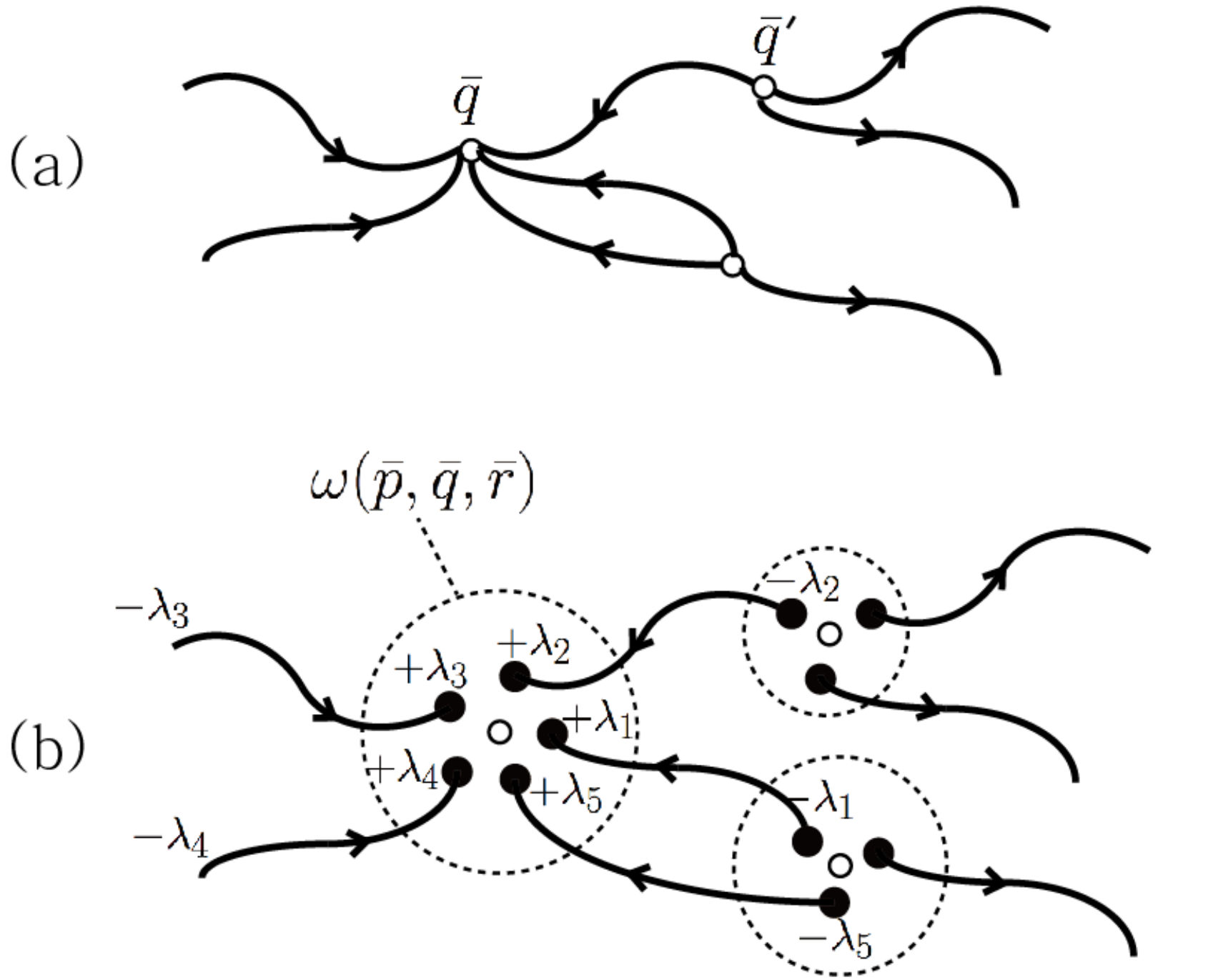}
\caption{Shape of 1-dimensional moduli space near an orientable critical points (a) topologically  and (b) considering orbifold structures}\label{moduli}
\label{1moduli}
\end{center}
\end{figure}

Near an non-orientable critical point, additional cancellation phenomena  as in Lemma \ref{cancel}
occur, and this will be explained in \eqref{eq:can}.

Now, we prove our main theorem \ref{thm:main}

\begin{proof}

Observe that
$$\nu_{\bar{r}} (\bar{\delta}) = \epsilon (\bar{\delta}) \dfrac{ |G_{\bar{r}}|}{|G_{\bar{\delta}}|} = \dfrac{ |G_{\bar{r}}|} {|G_{\bar{q}}|} \cdot \nu_{\bar{q}} (\bar{\delta}).$$
Therefore,
\begin{eqnarray*}
\partial^2 \bar{p} &=& \partial \left(\sum_{\bar{q} \in \CR_{k-1}^+ (\bar{f}) } \sum_{\bar{\gamma} \in \mathcal{M} (\bar{p},\bar{q})  }\nu_{\bar{q}} (\bar{\gamma}) \right) \\
&=& \sum_{\bar{r} \in \CR_{k-2}^+ (\bar{f})} \left( \sum_{\bar{q} }  \sum_{(\bar{\gamma}, \bar{\delta}) \in \partial \OL{\mathcal{M}} (\bar{p},\bar{r}) } \nu_{\bar{q}} (\bar{\gamma}) \nu_{\bar{r}} (\bar{\delta}) \right) \bar{r} \\
&=&  \sum_{\bar{r} \in \CR_{k-2}^+ (\bar{f})} |G_{\bar{r}}|   \left( \sum_{\bar{q} }  \sum_{(\bar{\gamma}, \bar{\delta}) \in \partial \OL{\mathcal{M}} (\bar{p},\bar{r}) } \dfrac{ \nu_{\bar{q}} (\bar{\gamma}) \nu_{\bar{q}} (\bar{\delta}) }{|G_{\bar{q} } |} \right) \bar{r} \\
&=& \sum_{\bar{r} \in \CR_{k-2}^+ (\bar{f})} |G_{\bar{r}}|  \left( \sum_{\bar{q} \in \CR_{k-1}^+ (f)} \omega (\bar{p}, \bar{q}, \bar{r}) \right) \bar{r} =0
\end{eqnarray*}
where the last sum is taken over all broken trajectories $(\bar{\gamma}, \bar{\delta})$ through $\bar{p}$, $\bar{q}$ and $\bar{r}$. The last equality follows from the lemma below, which directly implies the theorem.
\end{proof}

\begin{lemma}
 If $\bar{q}$ is orientable, then
$$  \sum_{ (\bar{\gamma}, \bar{\delta}) } \dfrac{ \nu_{\bar{q}} (\bar{\gamma}) \nu_{\bar{q}} (\bar{\delta}) }{|G_{\bar{q}}|}
=  \omega(\bar{p},\bar{q},\bar{r}),$$
and if $\bar{q}$ is non-orientable, then $\omega(\bar{p},\bar{q},\bar{r}) =0.$
Therefore,
\begin{equation}\label{eq:can}
\sum_{\bar{q} \in \CR_{k-1} (\bar{f})} \omega(\bar{p},\bar{q},\bar{r}) = \sum_{\bar{q} \in \CR_{k-1}^+ (\bar{f})} \omega(\bar{p},\bar{q},\bar{r})  =0.
\end{equation}
\end{lemma}

\begin{proof}
From \ref{brk}, $\omega(\bar{p},\bar{q},\bar{r})$ is equivalent to the ``weighted" number of elements of the space $\bigcup \, \Gamma \times \Delta / G_{\bar{q}}$ for all broken trajectories $(\bar{\gamma}, \bar{\delta})$ through $\bar{q}$, where the weight $[\gamma, \delta] \in \Gamma \times \Delta / G_{\bar{q}}$ is given by $\frac{\epsilon [\gamma,\delta] }{ |G_{[ \gamma, \delta ]}|}$. Thus, the weighted number of elements in $\Gamma \times \Delta / G_{\bar{q}}$ is

\begin{equation*}
\dfrac{1}{|G_{\bar{q}}|} \sum_{ (\gamma, \delta ) \in \Gamma \times \Delta } \dfrac{\epsilon [\gamma, \delta]}{|G_{[ {\gamma}, {\delta}]}|} \cdot | \{ g \in G_{\bar{q}} | g \cdot \gamma = \gamma, g \cdot \delta = \delta \}|
\end{equation*} 
 by Lemma \ref{count}. Since $| \{ g \in G_{\bar{q}} | g \cdot \gamma = \gamma, g \cdot \delta = \delta \} | = |G_{\gamma} \cap G_{\delta}| =  |G_{[ {\gamma}, {\delta}]}|$, it equals to
 $$\frac{1}{|G_{\bar{q}}|}  \sum_{ (\gamma, \delta ) \in \Gamma \times \Delta } \epsilon [\gamma, \delta]
 =\frac{1}{|G_{\bar{q}}|}  \sum_{ (\gamma, \delta ) \in \Gamma \times \Delta } \epsilon(\gamma) \cdot \epsilon(\delta) .$$
 If $\bar{q}$ is orientable, $\epsilon(\bar{\gamma}) \cdot \epsilon (\bar{\delta})$ is constant for all $(\gamma, \delta) \in \Gamma \times \Delta$ so that 
 
\begin{eqnarray*}
\dfrac{1}{|G_{\bar{q}}|}  \sum_{ (\gamma, \delta) \in \Gamma \times \Delta  } \epsilon [\gamma, \delta]
&=&
\dfrac{\epsilon (\bar{\gamma})\epsilon (\bar{\delta})}{|G_{\bar{q}}|}  \sum_{ (\gamma, \delta ) \in \Gamma \times \Delta}  1
\\
 &=& \dfrac{\epsilon (\bar{\gamma})\epsilon (\bar{\delta})}{|G_{\bar{q}}|} \cdot |\Gamma \times \Delta| \\
&=& \dfrac{\epsilon (\bar{\gamma})\epsilon (\bar{\delta})}{|G_{\bar{q}}|} \cdot \dfrac{ |G_{\bar{q}}| }{|G_{\bar{\gamma}}|} \cdot \dfrac{ |G_{\bar{q}}| }{|G_{\bar{\delta}}|} \\
&=&  \dfrac{ \nu_{\bar{q}} (\bar{\gamma}) \nu_{\bar{q}} (\bar{\delta}) }{|G_{\bar{q} }|} .
\end{eqnarray*}
Therefore, $\omega(\bar{p},\bar{q},\bar{r}) = \sum_{ (\bar{\gamma}, \bar{\delta}) } \dfrac{ \nu_{\bar{q}} (\bar{\gamma}) \nu_{\bar{q}} (\bar{\delta}) }{|G_{\bar{q}}|}$, if $\bar{q}$ is orientable.

On the other hand, suppose $\bar{q}$ that is non-orientable. Pick any $g \in G_{\bar{q}}$ which reverses the orientation of $W^-(p)$. Then, $g$ gives a permutation $\Gamma \times \Delta$ by
$g \cdot (\gamma, \delta) := (g \cdot \gamma, \delta).$ Note that
$$\epsilon (g \cdot \gamma) \cdot \epsilon( \delta) = -\epsilon (\gamma) \cdot \epsilon( \delta). $$
By the same argument in the case of global quotients (Lemma \ref{cancel}), the number of elements in $\Gamma \times \Delta$ which have positive signs should agree with the number of elements with negative signs. Thus, $ \sum_{ (\gamma, \delta ) \in \Gamma \times \Delta } \epsilon(\gamma) \cdot \epsilon(\delta)=0 $ and $\omega(\bar{p},\bar{q},\bar{r}) =0$ when $\bar{q}$ is not orientable.
\end{proof}

\begin{lemma}\label{count}
Let $S$ be a finite set on which a finite group $G$ acts. Suppose that $S/G$ is a weighted set such that each element $[x] \in S/G$ has the weight $\lambda_{[x]}$. Then,
$$\sum_{[x] \in S/G} \lambda_{[x]} = \dfrac{1}{|G|} \sum_{x \in S} \lambda_{[x]} \cdot |G_x|.$$
\end{lemma}

\begin{proof}
The lemma directly follows from the standard proof of the Burnside's lemma.
\end{proof}

\begin{remark}
The proof of $\underline{\partial}^2=0$ is similar. Indeed, this is automatic since we have a $(\partial, \underline{\partial})$-chain map $\psi : \bar{p} \mapsto |G_{\bar{p}}| \cdot \bar{p}$ which is an $\RR$-linear isomorphism.
\end{remark}

As a final remark, we propose another possible way of constructing a chain complex related to $\bX$. Recall that we have the $G$-equivariant Morse-Bott chain complex for $\tilde{f}:M \to \RR$ by lifting $\bar{f} : \bX \to \RR$, where $[M/G] \cong \bX$ and $G$ is a compact connected Lie group in general. The construction in \cite{AB} provides the Morse-Bott chain complex of $\tilde{f}$ equipped with the $G$-action, if the assumptions of \cite{AB} are met, such as
\begin{itemize}
\item[-] critical submanifolds are orientable;
\item[-] $G$-action preserves orientations of unstable and stable manifolds;
\item[-] evaluation maps from the trajectories are submersions. 
\end{itemize}
Even when all these conditions are met,
the complex in \cite{AB} uses Cartan model of $BG$, and it is not clear what the relation of the differential there and geometric counting of gradient flow-lines in $\bX$ is.
For $\RR$ or $\QQ$-coefficients, the $G$-equivariant cohomology of $M$ in this setting is isomorphic to the singular homology of $M/G$ \cite[Proposition 2.12]{ALR}. Thus it would be interesting to find a
relation between the construction of \cite{AB} and the construction given here in the above setting.

\section{Comparison on the Morse and the singular homology of the quotient space}\label{sec:comparison}
 In this section, we show that the homology of the Morse complex of general orbifolds $(CM_*(\bX, \bar{f}), \partial)$ equals the singular homology of the quotient space.
 We assume in this section that $\bar{f}$ is self-indexing, meaning that $\bar{f} (\bar{p}_i) = \lambda_i$ where $\lambda_i$ is the Morse index of $\bar{p}_i$. 

\begin{remark}
For the general case without self-indexing assumption, one may use the filtration
$$X_k := \displaystyle\bigcup_{{\rm{ind}} (\bar{p}) \leq k} \underline{W}^- (\bar{p}),$$
$$ \phi=X_{-1} \subset X_0 \subset X_1 \subset \cdots \subset X_{n-1} \subset X_n = X$$
instead of the one described below and proceed as in \cite{S}, where $\underline{W}^- (\bar{p})$ is given in \eqref{undW}. Note that $X_k$ is compact. 
\end{remark}

We will apply the topological method of \cite{Ni} which uses the cell structure of $\bX$ induced by Morse data of $f$. This kind of a cell structure was already revealed by several authors, for example \cite{LT} and \cite{H}.

\begin{theorem} \cite[Theorem 7.6]{H} Let $\bar{p} \in \CR_k(\bar{f})$ and $\bar{f} (\bar{p})=c$. Suppose that $\bar{p}$ is the unique critical point in $\bar{f}^{-1} [c- \epsilon , c+ \epsilon]$ for small $\epsilon \ll \frac{1}{2}$. Then, $\bar{f}^{-1} (-\infty, c+\epsilon]$ is homotopic to  $\bar{f}^{-1} (-\infty, c-\epsilon]$ attached with $D^k / G_{\bar{p}}$ along $\partial D^k / G_{\bar{p}}$. Here, $D^k$ is a small invariant disc in the unstable manifold at $\bar{p}$ in a uniformizing chart around $\bar{p}$, and hence endowed with the $G_{\bar{p}}$-action.
\end{theorem}

\begin{proof}
See \cite[Theorem 7.6]{H} and compare it with \cite[3.2]{M}.
\end{proof}

We need an elementary fact of equivariant topology to compute homological information of attaching cells.

\begin{theorem} \cite[Theorem 2.4]{BR} \label{br}
Let $K$ be a regular $G$-simplicial complex with $G$ finite and $L$ be a subcomplex. Then,
$$H_\ast (K,L ; \RR)^G \cong H_\ast (K/G, L/G ; \RR),$$
where the left hand side means the subset of $H_\ast (K,L ; \RR)$ fixed by $G$.
\end{theorem}

\begin{corollary}\label{disk}
Let $D^n$ be the $n$-dimensional disc and the finite group $\GG$ act on $(D^n, \partial D^n)$. Then, the homology group $H_\ast (D^n / \GG, \partial D^n / \GG ;\RR)$ are given as follows.
\begin{equation*}
H_i (D^n / \GG,\partial D^n/ \GG ;\RR)=
\left\{
\begin{array}{ll}
\RR & \mbox{ if}\,\,\, i=n \mbox{ and} \,\,\,\GG \mbox{ preserves
 the orientation of} \,\,\,D^n\\
0 & \, \mbox{otherwise} 
\end{array}.\right.
\end{equation*}
\end{corollary} 

\begin{proof}
It suffices to note that there exists a $\GG$-invariant triangulation of $(D^n, \partial D^n)$ by \cite{IL} and that we can achieve the regularity condition in Theorem \ref{br} by suitable subdivisions.
\end{proof}

Recall that in smooth case,  the coefficient of $q$ in $\partial p$ is defined by the (relative) intersection number between the unstable of $p$ and the stable manifold of $q$ (see \cite{Ni}). In what follows, we will use the integration of Thom forms instead as they are more suitable in the orbifold setting.
Recall from \cite{CR} (or \cite{ALR}) that the Thom form of a suborbifold $\bold{N}$ of $\bX$ is defined locally as the invariant Thom form of the preimage $\WT{N}$ of $\bold{N}$ in each uniformizing chart. Let $N$ denote the underlying quotient space of $\bold{N}$.

\begin{remark}\label{Thom}
On an uniformizing chart, integration of Thom form along a normal fiber of $\WT{N}$ at $p \in \WT{N}$ is $1$ where $ \pi (p) = \bar{p}$, according to the usual definition of Thom forms on Euclidean spaces.
See \cite{ALR} for more details about Thom forms and Poincare duals of suborbifolds.
\end{remark}

We will also need the Stokes theorem for orbifolds, which goes back to \cite{Sa}. We  recall it here for readers convenience. A $C^\infty$ singular simplex $\bar{s}$ of dimension $k$ in $\bX$ is defined by a smooth map $\bar{s}$ from a $k$-dimensional simplex $\Delta_k $ to $X$. Suppose that the image of $\bar{s}$ lies in a single uniformizing chart $(\WT{U}, G, \pi)$ and it admits a lifting $s : \Delta_k \to \WT{U}$ with $ \pi \circ s = \bar{s}$.
Consider  a $k$-form $\bar{\omega}$ on $\pi(\WT{U})$, which is given by an invariant 
$k$-form $\omega$ on $\WT{U}$. We define
\begin{equation}\label{intoverchain}
\int_{\bar{s}} \bar{\omega} = \int_{\Delta_k} s^\ast \omega.
\end{equation}
For a general $\bar{s}$, we use a partition of unity to define $\int_{\bar{s}} \bar{\omega} $. One can check that the Stokes formula still holds:
$$\int_{\bar{s}} d \bar{\omega} = \int_{\partial \bar{s}}  \bar{\omega}.$$

\begin{remark}
Here, we think of $s$ as a singular chain defined on the usual domain in the Euclidean space and hence, there is no weight in the integral \eqref{intoverchain}. 

Since $s$ is given by the composition $\pi \circ s$ and $\pi$ is a quotient map, $s$ could wrap the image several times. In the proof of the next proposition, we will divide such a singular chain by the multiplicity of $\pi$, which explains weights in front of the integrals. This is basically the idea behind the definition of the orbifold integration in \cite{ALR}. Integrals that appear below are not orbifold integrals.
\end{remark}

\begin{prop}\label{comp}
The homology of $(CM_*(\bX, \bar{f}), \partial)$ constructed in {\rm Section 4} equals the singular homology of the underlying space $X$ of $\bX$. 
\end{prop}

\begin{proof}
We begin with a filtration of singular homology of $X$. Let $X_k := \bar{f}^{-1} (-\infty, k+1- \epsilon)$, $0<\epsilon \ll 1$, and $Y_k := \bar{f}^{-1} [k-\epsilon, k+1-\epsilon]$. Then,
\begin{equation}\label{filt}
C_\ast ( X_0 ; \RR ) \subset C_\ast ( X_1 ;\RR ) \subset \cdots \subset C_\ast ( X_n ; \RR )= C_\ast ( X ; \RR)
\end{equation}
gives a filtration on the singular chain complex $C_\ast ( X;\RR)$. For a critical point $\bar{p}$ with $\bar{f}(\bar{p})=k$, let $\underline{W}^\pm (\bar{p})$ be the stable and unstable manifolds at $\bar{p}$, respectively. i.e.
\begin{equation}\label{undW}
\underline{W}^\pm (\bar{p}) = \{ \bar{x} \in X : \lim_{t \to \pm \infty} \bar{\Phi}_t (\bar{x}) = \bar{p} \}.
\end{equation}
Set $\underline{D}^\pm (\bar{p}) = \underline{W}^\pm (\bar{p}) \cap Y_k$. Then, topologically (since $\epsilon$ is small enough)
$$\underline{D}^\pm (\bar{p}) \cong D^\pm (p) / G_{\bar{p}},$$
where $D^\pm (p)$ are small invariant neighborhoods of $p \in \pi^{-1} (\bar{p})$ in stable and unstable manifolds of $p$ with respect to the lift $f$ of $\bar{f}$. By $\partial \underline{D}^\pm (\bar{p})$, we mean the image of $\{ \partial D^\pm (p) \} / G_{\bar{p}}$, or equivalently 
$$\partial \underline{D}^+ (\bar{p}) = \underline{D}^+ (\bar{p}) \cap \{ \bar{f} = k+1-\epsilon \},$$
$$\partial \underline{D}^- (\bar{p}) = \underline{D}^- (\bar{p}) \cap \{ \bar{f} = k-\epsilon \}.$$
By the excision, we have (see  \cite{H})
\begin{equation*}
H_\ast (X_k, X_{k-1} ;\RR)=
\left\{
\begin{array}{ll}
\displaystyle\bigoplus_{\bar{p} \in \CR_k (\bar{f})} H_k (\underline{D}^- (\bar{p}) , \partial \underline{D}^- (\bar{p})   ; \RR) \,\, & \ast=k \\
\qquad \qquad \qquad \,\,\,\,\, 0 & \, \mbox{otherwise} 
\end{array}.\right.
\end{equation*}
From \ref{disk}, $H_k (\underline{D}^- (\bar{p}) ,  \partial \underline{D}^- (\bar{p}) ; \RR) \cong H_k (D^- (p), \partial D^- (p) )^{G_{\bar{p}} }$ is isomorphic to $\RR$  if $\bar{p}$ is orientable and vanishes otherwise.

The $E^1$-terms of the spectral sequence coming from \eqref{filt} produce the following chain complex:
$$ \cdots \to H_{k+1} (X_{k+1}, X_{k} ;\RR)  \to  H_k (X_k, X_{k-1} ;\RR) \to H_{k-1} (X_{k-1}, X_{k-2} ;\RR) \to \cdots ,$$
where the boundary map is given by the composition
$$ H_{k} (X_{k}, X_{k-1} ;\RR)  \to  H_{k-1} (X_{k-1} ;\RR) \to H_{k-1} (X_{k-1}, X_{k-2} ;\RR).$$

\begin{figure}[h]
\begin{center}
\includegraphics[height=3.3in]{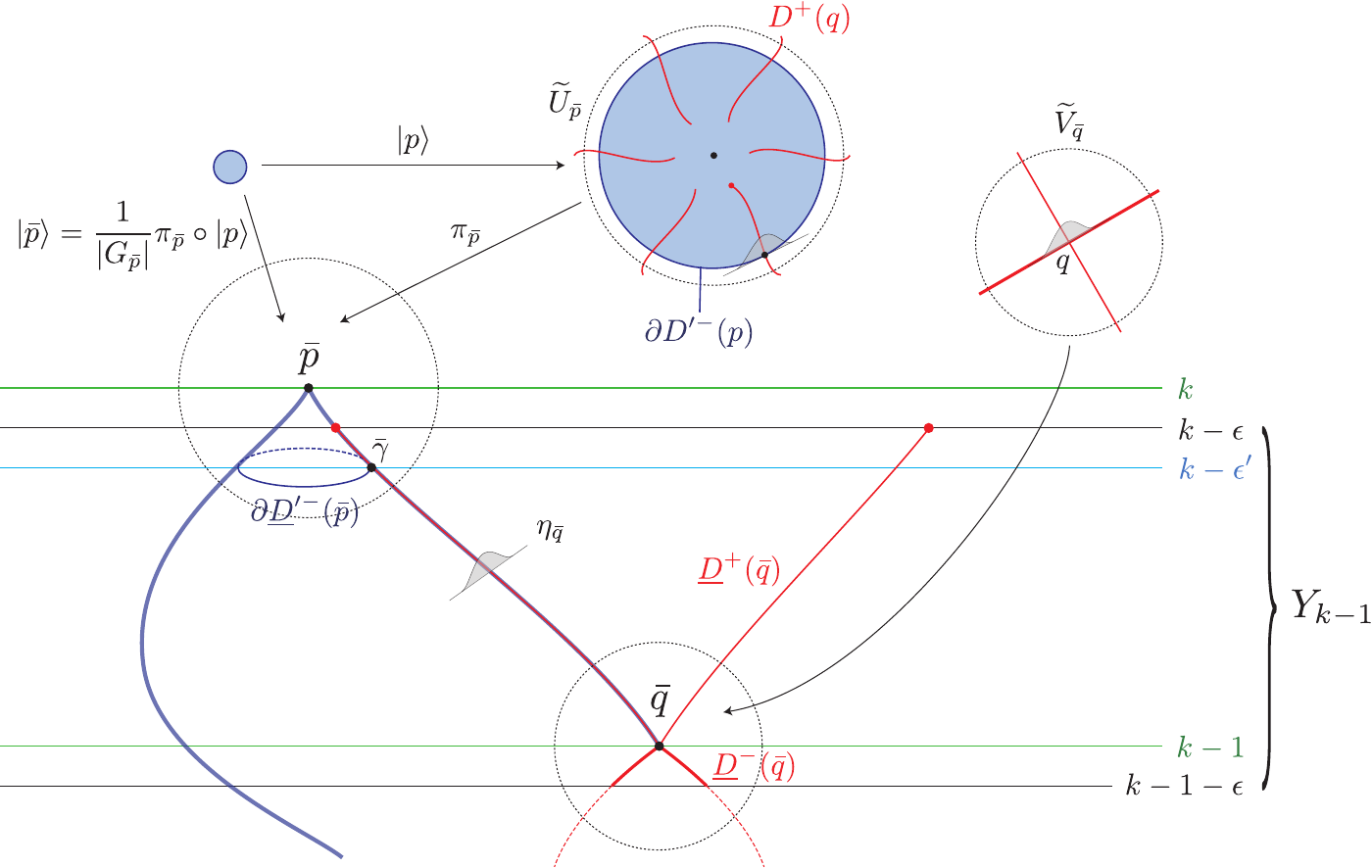}
\caption{CW structure of the quotient space}
\label{cellstr}
\end{center}
\end{figure}

We now choose a generator of $H_{k} (X_{k}, X_{k-1} ;\RR) $ which is roughly a smooth singular chain in $X_k$ representing $$(\underline{D}^- (\bar{p}) ,  \partial \underline{D}^- (\bar{p}) ) \subset (X_{k}, X_{k-1} ) $$ for $\bar{p} \in \CR_k^+ (\bar{f})$. For sufficiently small $\epsilon$, we may assume that the set $\underline{D}^- (\bar{p})$ lies entirely in the uniformizing neighborhood $\pi_{\bar{p}} (\WT{U}_{\bar{p}})$ (for $(\WT{U}_{\bar{p}}, G_{\bar{p}}, \pi_{\bar{p}})$) around $\bar{p}$. Then, we have a neighborhood $D^- (p)$ of $p=\pi_{\bar{p}}^{-1} (p)$ in the unstable manifold of the lifting of $f$ which covers $\underline{D}^- (\bar{p})$ downstairs. $D^- (p)$ obviously represents a smooth singular chain in $\WT{U}_{\bar{p}}$, which is a map $\left| p \right>$ from (formal sum of) simplices in the Euclidean space to $\WT{U}_{\bar{p}}$. Note that the composition $\pi_{\bar{p}} \circ \left| p \right>$ wraps the original set $\underline{D}^- (\bar{p})$ $|G_{\bar{p}}|$-times (if $G_{\bar{p}}$-action on $D^- (p)$ is effective). Thus, it is natural to define
\begin{equation}\label{genE1}
\left| \bar{p} \right> = \frac{1}{|G_{\bar{p}}|} \pi_{\bar{p}} \circ \left| p \right>
\end{equation}
and we will use $\left| \bar{p} \right>$ as fixed generator of $H_{k} (X_{k}, X_{k-1} ;\RR) $. (See Figure \ref{cellstr}.)  We do the similar for each $\bar{q} \in \CR^+ (\bar{f})$.

Then, there exists a certain real number $a_{\bar{q}}$ for each $\bar{q} \in  \CR_{k-1}^+(\bar{f})$ such that
\begin{equation}\label{ket}
\partial \left| \bar{p} \right> = \sum_{ \bar{q} \in \CR_{k-1}^+(\bar{f})} a_{\bar{q}} \, \left| \bar{q} \right>.
\end{equation}
Now, it is enough to show that 
$$a_{\bar{q}} = n (\bar{p}, \bar{q}) = \sum_{\bar{\gamma} \in \mathcal{M} (\bar{p},\bar{q}) } \epsilon(\bar{\gamma}) \dfrac{|G_{\bar{q}}|}{|G_{\bar{\gamma}}|} .$$

For this, we will consider smooth singular chains $\partial \left| \bar{p} \right>$ and $\left| \bar{q} \right>$ in the subspace
$Y_{k-1}$ of $X$, and use the Thom form of $\underline{D}^+(\bar{q})$ to identify the constant $a_{\bar{q}} = n (\bar{p},\bar{q})$.
Denote by $\eta_{\bar{q}}$ the Thom form of $\underline{D}^+(\bar{q})$ on $Y_{k-1}$.

Set $\underline{D}'^- (\bar{p}) := W^-(\bar{p}) \cap \bar{f}^{-1} [k-\epsilon', \infty)$ and $\partial \underline{D}'^- (\bar{p}) := \underline{D}^- (\bar{p}) \cap \bar{f}^{-1} (k-\epsilon')$ for $ \epsilon < \epsilon'$ and write $\left| p \right>'$  for the singular chain in $\WT{U}_{\bar{p}}$ representing $D'^- (p)$ which lies over $\underline{D}'^- (\bar{p})$. (we also assume $\epsilon'$ to be so small that there exists an uniformizing chart $(\WT{U}_{\bar{p}}, G_{\bar{p}}, \pi_{\bar{p}})$ around $\bar{p}$ with $\partial \underline{D}'^- ( {\bar{p}} ) \subset \pi_{\bar{p}} (\WT{U}_{\bar{p}})$.)

Define
$$\partial \left| \bar{p} \right>' := \frac{1}{|G_{\bar{p}}|} \pi_{\bar{p}} \circ \partial \left| p \right>'$$
and identify $\partial \left| \bar{p} \right>$ with $\partial \left| \bar{p} \right>'$
by flowing down $\partial \left| \bar{p} \right>$ along negative gradient flows from $\bar{f}^{-1}(\epsilon)$ to
 $\bar{f}^{-1}(\epsilon')$. (We made such a deformation not to place possible intersections of
 $\partial \underline{D}'^- (\bar{p})$ and $\underline{D}^+(\bar{q})$ on the boundary of $Y_k$.)

As \eqref{ket} holds on the homology level, we can take a formal sum of simplicial complexes $K$ which maps (say, via $\tau$) to  $Y_{k-1}$, whose boundary $\partial \tau:\partial K \to Y_{k-1}$ is given by 
$$\partial \left| \bar{p} \right>' \cup \bigcup_{\bar{q}} a_{\bar{q}} \left| \bar{q} \right>,$$
with the opposite orientation on the first component (relative to $X_{k-2}$).

 Here, we consider  $\tau:K \to Y_{k-1}$, 
 $\partial \left| \bar{p} \right>'$ and $\left| \bar{q} \right>$ as (smooth) singular chains on $ Y_{k-1}$.  The rational coefficients are allowed as we work with $\RR$-coefficients for the 
 singular homology. By subdividing simplices repeatedly if necessary, we may assume that the map $\tau$ when restricted to each simplex in $K$ has a lift in some uniformizing chart of $\bX$.
Then, the Stokes theorem of \cite{Sa} tells us that
$$\int_{\partial K} \tau^\ast \eta_{\bar{q}} = \int_K d( \tau^\ast  \eta_{\bar{q}})  =  \int_K \tau^\ast (d \eta_{\bar{q}}) =0$$

 We will compute the integral $\int_{\partial K} \tau^\ast \eta_{\bar{q}} $ to get $a_{\bar{q}}$. Since the support of $\eta_{\bar{q}}$ can be shrunken so that it lies in an arbitrary small open neighborhood of $\underline{D}^+(\bar{q})$, (and since for each $\bar{q}$,  $\underline{D}^+(\bar{q})$'s are disjoint,)
\begin{equation}\label{IJ}
\displaystyle\int_{\partial K} \tau^\ast \eta_{\bar{q}} = - \int_I \tau^\ast \eta_{\bar{q}} \,\,  + \,\, \int_ J \tau^\ast \eta_{\bar{q}}\,\,(=0),
\end{equation}
where $I$ and $J$ are sub-complexes of $\partial K$ mapping to $\partial \underline{D}'^- ( {\bar{p}} )$ and $a_{\bar{q}} \underline{D}'^- ({\bar{q}})$ (via $\partial \left| \bar{p} \right>'$ and $a_{\bar{q}} \left| \bar{q}\right>$), respectively. ($ a_{\bar{q}}$ is considered to be a coefficient of a singular chain.)

Observe that the preimage of $\underline{D}^+(\bar{q})$ in $\WT{U}_{\bar{p}}$ (which is $D^+(q)$) will meet $\partial {D}'^- ( p )$, $\left( \sum_{\bar{\gamma} : \bar{p} \to \bar{q} }\frac{ \epsilon(\bar{\gamma}) |G_{\bar{p}}| }{ |G_{\bar{\gamma}} |} \right)$-times considering the orientation of intersection. (See the piture of $\WT{U}_{\bar{p}}$ in Figure \ref{cellstr}.) As mentioned earlier, $\ \sum_{\bar{\gamma} : \bar{p} \to \bar{q} }\frac{ \epsilon(\bar{\gamma}) |G_{\bar{p}}| }{ |G_{\bar{\gamma}} |} $ is nothing but the number of gradient flow lines starting at $p$ which lift flow lines in $X$ from $\bar{p}$ to $\bar{q}$.

Let a $G_{\bar{p}}$-invariant differential form $\WT{\eta_{\bar{q}} }$ represent $\eta_{\bar{q}}$ on $\WT{U}_{\bar{p}}$. Recall from Remark \ref{Thom} that on each uniformizing chart intersecting $\underline{D}^+ (\bar{q})$, $\eta_{\bar{q}}$ is defined exactly in the same way as Thom forms in smooth cases.
Therefore, the integration of $\WT{\eta_{\bar{q}}}$ over $\left| \bar{p} \right>$ counts the number of intersection points of $\partial {D}'^- ( p )$ and $D^+(q)$. This combined with \eqref{intoverchain} and \eqref{genE1} implies,
\begin{eqnarray*}
\int_I \tau^\ast \eta_{\bar{q}} &=& \frac{1}{|G_{\bar{p}}|} \int_{\partial {D}'^- ( p ) } \WT{\eta_{\bar{q}} }\\
&=& \frac{1}{|G_{\bar{p}}|}  \cdot \left( \sum_{\bar{\gamma} : \bar{p} \to \bar{q} }\frac{ \epsilon(\bar{\gamma}) |G_{\bar{p}}| }{ |G_{\bar{\gamma}} |} \right) \cdot \int_{F} \WT{\eta_{\bar{q}} } \\
&=& \sum_{\bar{\gamma} : \bar{p} \to \bar{q} }\frac{ \epsilon(\bar{\gamma}) }{ |G_{\bar{\gamma}} |},
\end{eqnarray*}
 where $F$ denotes the general fiber of the normal bundle of $D^+(q)$ in $\WT{U}_{\bar{p}}$. Here, we use the transversality at intersection points of $\partial \underline{D}'^-(\bar{p})$ and $\underline{D}^+ (\bar{q})$.

On the other hand, in the uniformizing chart $\WT{V}_{\bar{q}}$ around $\bar{q}$, the preimage $D^- (q)$ of $\underline{D}^- (\bar{q})$ is used to calculate $\int_J \tau^\ast \eta_{\bar{q}}$ and get
\begin{eqnarray*}
\int_ J \tau^\ast \eta_{\bar{q}} &=&  \frac{ a_{\bar{q}}  }{|G_{\bar{q}}|} \int_{D^- (q)} \WT{\eta_{\bar{q}} } \\
&=& \frac{ a_{\bar{q}}  }{|G_{\bar{q}}|} \int_{F_q} \WT{\eta_{\bar{q}} } \\
&=&  \frac{ a_{\bar{q}}  }{|G_{\bar{q}}|},
\end{eqnarray*}
since $D^- (q)$ and $D^+ (q)$ meet only once at $q$.
(We abbreviate $\WT{\eta_{\bar{q}} }$ to denote the representative of $\eta_{\bar{q}}$ on $\WT{V}_{\bar{q}}$.)

By comparing both integrals and \eqref{IJ}, we conclude that

$$ a_{\bar{q}} =  \sum_{\bar{\gamma} \in \mathcal{M} (\bar{p},\bar{q}) } \epsilon(\bar{\gamma}) \dfrac{|G_{\bar{q}}|}{|G_{\bar{\gamma}}|}.$$
\end{proof}

\begin{remark}
Note that there are several points in the proof where we relied on the fact that $CM_\ast(\bX, \bar{f})$ is defined over the field coefficient, although it would be still a chain complex even if using $\ZZ$-coefficients.
\end{remark}

\begin{remark}
If we instead use $\pi_{\bar{p}} \circ \left| p \right>$ as a generator of $H_\ast (X_k, X_{k-1};\RR)$ (without dividing it by $|G_{\bar{p}}|$), then the resulting $E^1$-term is isomorphic to $\left( CM_\ast (\bX, \bar{f}), \underline{\partial} \right)$ \eqref{underpar}.
\end{remark}

Therefore, under the existence of a Morse-Smale function $\bar{f}$ on $\bX$, we can prove the Poincare duality of the singular homology of the quotient space $X$ by considering $-\bar{f}$. However, the inner product which gives the Poincare pairing between $HM_\ast (\bX, \bar{f})$ and $HM_\ast (\bX, -\bar{f})$ is induced by a slight different pairing
$$<\, \, \, ,  \, \, > :  CM_\ast (\bX, \bar{f}) \otimes CM_\ast (\bX, -\bar{f}) \to \RR,$$
from the usual one where $< \bar{p}, \bar{p}> = 1 / |G_{\bar{p}}|$ and $<\bar{p},\bar{q}> = 0$ if $\bar{p} \neq \bar{q}$. Let $\partial_+$ and $\partial_-$ be boundary operators of $CM_\ast (\bX, \bar{f})$ and $CM_\ast (\bX, -\bar{f}) $, respectively. Then,
\begin{eqnarray*}
< \partial_+ \bar{p} , \bar{q} >&=& n(\bar{p},\bar{q}) <\bar{q},\bar{q}> \\
&=& \dfrac{1}{|G_{\bar{q}}|} \sum_{\bar{\gamma} \in \mathcal{M} (\bar{p},\bar{q}) } \epsilon(\bar{\gamma}) \dfrac{|G_{\bar{q}}|}{|G_{\bar{\gamma}}|} \\
&=& \sum_{\bar{\gamma} \in \mathcal{M} (\bar{p},\bar{q}) } \dfrac{ \epsilon(\bar{\gamma})}{|G_{\bar{\gamma}}|},
\end{eqnarray*}
and similarly,
$$<\bar{p}, \partial_- \bar{q}> = n(\bar{q},\bar{p}) <\bar{p},\bar{p}> = \sum_{\bar{\gamma} \in \mathcal{M} (\bar{p},\bar{q}) }  \dfrac{\epsilon(\bar{\gamma}) }{|G_{\bar{\gamma}}|}$$
so that
$$ < \partial_+ \bar{p} , \bar{q} >= <\bar{p}, \partial_-\bar{q}>, $$
and hence $<\,\,\, , \,\,>$ induces a pairing on homologies.

\begin{remark}
To get a similar pairing between $(CM_\ast (\bX, \bar{f} ), \underline{\partial})$ and $ (CM_\ast (\bX, -\bar{f}) , \underline{\partial}) $, one should modify $<\,\,\, , \,\,>$ by $< \bar{p}, \bar{p}> = |G_{\bar{p}}|$.
\end{remark}

\section{Equivariant Transversality and weak group actions}\label{equiv_trans}
In this section, we discuss the problem of equivariant transversality, which in our case asks whether or not it is possible to make a generic $G$-invariant Morse function Morse-Smale.
For a global quotient orbifold $[M/G]$, we propose an alternative approach.
Namely, we consider a Morse-Smale function $f:M \to \RR$ which is {\em not} necessarily $G$-invariant. (This is possible of course since a generic function on $M$ is Morse-Smale.) Then, we define  what we will call a weak $G$-action on the Morse complex of $f$. The idea is to consider an $G$-equivariant family of functions
$\{ f \circ g^{-1} \}_{g \in G}$ and find quasi-isomorphisms between the Morse complexes of them using continuation maps. This induces an honest $G$-action on the Morse homology of $M$, and we show that if
we start with a different Morse-Smale function, then we get an invariance in the weak sense.
\subsection{Perturbations to $G$-Morse-Smale functions for global quotients}
We explain the problem on the genericity of Morse-Smale condition for $G$-invariant functions.
A well-known way of proving that a generic function is Morse with Morse-Smale condition is as follows.
(See \cite{AB} for example.)
Consider a compact manifold $M$ and the space of smooth functions $C^\infty(M)$ with the map
$$\Psi: M \times C^\infty(M)  \to T^*M: \;\; (x, f) \mapsto df_x.$$
The function $f$ is Morse if $M \to T^*M$, given by $\Psi( \cdot, f)$, is transverse to the zero section.
In other words, when $df(x)=0$ we need the Hessian  $\textrm{Hess}(f)_x:T_xM \to T_x^*M$ to be an isomorphism.
One can prove that $\Psi$ is transverse to the zero section as a map from $M \times C^\infty(M)$. This implies that there is a Baire set $\mathcal{B} \subset
C^\infty(M)$ of Morse functions on $M$. 

To show the genericity of Morse condition in the existence of a $G$-action,, 
we need to show that  the restriction 
$$\Psi_G: M \times (C^\infty(M))^G  \to T^*M$$
is also transverse to the zero section. 
Note that for a point $x$ with the isotropy group $G_x$, we can define a $G_x$-invariant distance function from $x$ on a neighborhood of $x$, which can extend to a $G$-invariant function on $M$. Since distance functions have non-degenerate
Hessians, one can show that $\Psi_G$ is transverse to the zero section, and there is a Baire set of
$G$-invariant Morse functions on $M$ (see \cite{Wa}).

For the Morse-Smale condition without considering $G$-actions, \cite[Proposition B.2]{AB} proves that there is a
Baire set of Morse functions with Morse-Smale gradient flows, and the general scheme of the proof goes as follows.
If a gradient flow of $x$ and a negative gradient flow of $y$ meet at a point $b$, one tries to find a 
perturbation of a gradient flow of $x$ (or the one of $y$) to another direction $v \in T_bM$.
This is done by considering a suitable vector field $Y$ along a gradient flow of $x$,
and then, find some function $\widetilde{f}$ whose gradient flow restricts to $Y$.

However, if $Y$ is not a $G$-invariant vector field, then,  it is impossible
to find a $G$-invariant function $\widetilde{f}$ with the gradient flow $Y$. 
Thus such a perturbation scheme does not work in the presence of $G$-action, and apparently the Morse-Smale condition for $G$-invariant functions is not generic.

A close inspection shows that if non-trivial gradient flows always lie on the smooth part $M^{sm}$ of $M$, then
the same argument is still valid to show the following.
\begin{lemma}
If $M^{sing}=M \setminus M^{sm}$ is the set of isolated points, then there is
a Baire set of $G$-invariant Morse functions with Morse-Smale gradient flows.
For example, the assumption holds for any surface $M$ with an orientation preserving $G$-action on it.
\end{lemma}
 
\subsection{Equivariant families of Morse-Smale functions and weak group actions}
Next, we consider a  Morse-Smale function $h:M \to \RR$, which is not $G$-invariant.
Our original motivation to work with such a function $h$ was that we can obtain such a function $h$ by perturbing a given $G$-invariant Morse function $f$, but the resulting  function $h$ may not be $G$-invariant anymore.
But in the rest of the section, we do not assume that $h$ and $f$ are close to each other, and hence $h$ can be an arbitrary Morse-Smale function on $M$
which forms a dense subset of $C^\infty(M)$.

We consider a $G$-equivariant family of Morse-Smale functions $$\mathcal{F} := \{h_g:=h \circ g^{-1} \mid g \in G\},$$
where $\mathcal{F}$ has a left $G$-action.
Note that if $x \in crit(h)$, then $g \cdot x \in crit(h_g)$.
For simplicity, we assume that $G$-action on $\mathcal{F}$ is free, which means that $h_g=h$ only when $g$ is the identity element of $G$
(general cases can be handled by considering both strict and weak group actions).
Since $h$ is not $G$-invariant, we can not define $G$-action on $CM_\ast (M,h; \Theta)$ directly as done in \eqref{CMTheta}, Section 2.

We first introduce the notion of weak $G$-actions on chain complexes.
\begin{definition}
Let $(C_\ast, \partial)$ be a chain complex over $\mathbb{R}$. We say $C_\ast$ admits a weak $G$-action if the following
holds.
\begin{enumerate}
\item For each $g \in G$ and $x \in C_\ast$, we have $g \cdot  \partial x = \partial (g \cdot x)$.
\item For each $g,k \in G$, there exists a chain homotopy $\sigma_{g,k}$ from $C_\ast$ to itself satisfying
$$g \cdot \left( k \cdot x \right)  - (gk) \cdot x = \sigma_{g,k} \, \partial + \partial \, \sigma_{g,k}.$$
\end{enumerate}
\end{definition}
In particular, a weak $G$-action on $C_\ast$ obviously induces an honest $G$-action on its homology $H_\ast(C)$. 

In order to define a weak $G$-action on the Morse chain complex $CM_\ast (M,h; \Theta)$ of $h$,  We use a continuation map induced by a homotopy $H_g$ from $h_g$ to $h$ for each $g \in G$ such that
\begin{equation}\label{continuationhpty}
H_g (x,t) = \left\{
\begin{array}{ll}
h_g (x) &  t \leq 0\\
h (x) &  t \geq 1
\end{array}.
\right.
\end{equation} 
This induces a chain map called the continuation 
\begin{equation}\label{phig}
\phi_g : CM_\ast (M, h_g;\Theta) \to CM_\ast (M, h;\Theta).
\end{equation}
which counts flow lines $\gamma$ satisfying $\gamma'(t) = -\nabla H ( \gamma(t), t)$ as shown in Figure \ref{contflow}, below. Here, $\nabla H ( x, t)$ is the unique vector satisfying
$$dH(x, t) (v) = \left<\nabla H (x,t), v \right>$$
for all $v \in T_x M$ and for a $G$-invariant metric $\left<, \right>$.
\begin{figure}[htp]
\begin{center}
\includegraphics[scale=0.4]{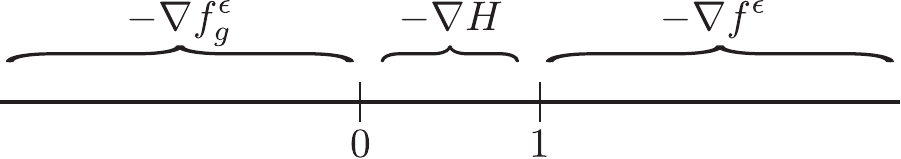}
\caption{flow lines counted for continuation homomorphism}\label{contflow}
\end{center}
\end{figure}
 It is well known from the standard Morse theory that $\phi_g$ is a chain map which induces an isomorphism on the homology level. (See for e.g. \cite{Hu})  
 
 Moreover, two continuation homomorphisms from different homotopies $H_g$ and $H_g'$ are always chain homotopy equivalent to each other. i.e. for $\phi'_g : CM_\ast (M,h_g;\Theta) \to CM_\ast (M,h;\Theta)$ from another homotopy $H'_g$ between $h_g$ and $h$, there exists a chain homotopy $\sigma$ on $CM_\ast (M,h;\Theta)$ such that $\phi'_g - \phi_g =  \sigma \partial + \partial \sigma$. ($\sigma$ is given by a homotopy between homotopies $H_g$ and $H_g'$.)


\begin{prop}\label{weakGonCM}
For $p \in crit (h)$ and $g \in G$, define $$g \cdot_{w} p = \phi_g (g \cdot p) \in CM_\ast (M,h;\Theta).$$ Then, $(\cdot_w)$ gives rise to a weak $G$-action on $CM_\ast (M,h;\Theta)$.
\end{prop}

\begin{proof}
We write $g(p) \in crit(h_g)$ for the $G$-action on $p \in crit(h)$ in the proof to avoid a confusion with an weak action.
The first property of the weak $G$-action follows since continuation map $\phi_g$ is a chain map.

As $g \cdot_w (k \cdot_w p) = \phi_g \, g \, \phi_k \, k (p)$ and $(gk) \cdot_w p = \phi_{gk} \, gk (p)$, we need to find a chain homotopy $\sigma_{g,k}$ between two chain maps $\phi_g \, g  \, \phi_k \, k$ and $\phi_{gk} \, gk$. Consider a chain map 
$$g \, \phi_k \, g^{-1} : CM_\ast (M ,h_{gk} ;\Theta) \to CM_\ast (M, h_{g};\Theta)$$
defined by the following commutative diagram.
\begin{equation}
\xymatrix{ CM_\ast (M, h_{gk};;\Theta) \ar[d]_{g^{-1}} \ar[r]^{g \, \phi_k \, g^{-1}} & CM_\ast (M, h_{g};\Theta) \\
CM_\ast (M, h_{k};\Theta)  \ar[r]_{\phi_k} & CM_\ast (M, h;\Theta)  \ar[u]_g
}
\end{equation}

We claim that that $g\, \phi_k g^{-1}$ is again a continuation homomorphism which is given by a homotopy $H_k \circ g^{-1}$  between $h_{gk}$ and $h_g$. First of all,
\begin{equation*}
H_k \circ g^{-1} (x,t) = \left\{
\begin{array}{ll}
h_k (g^{-1} \cdot x) = h_{gk} (x) &  t \leq 0\\
h (g^{-1} x) = h_g (x)  &  t \geq 1
\end{array}
\right.
\end{equation*}
by the definition of $H_k$ \eqref{continuationhpty}. Suppose that $\gamma$ is a flow line counted for $\phi_k$ (which is of the shape as in Figure \ref{contflow}) and it flows from $k \cdot p \in crit(h_k)$ to $p \in crit(h)$. Since
\begin{eqnarray*}
d( H_k \circ g^{-1}) (y, t) (v) &=& dH_k (g^{-1} \cdot y, t) ((g^{-1})_\ast v) \\
&=&  \left<\nabla H (g^{-1} \cdot y, t), g^{-1}_\ast v \right> \\
&=& \left< g_\ast \nabla H (g^{-1} \cdot y, t), v \right>
\end{eqnarray*}
for $v \in T_x M$ by $G$-invariance of the Riemannian metric, we have $g_\ast \nabla H (x, t) = \nabla (H_k \circ g^{-1}) (g \cdot x,t)$ by letting $x= g^{-1} \cdot y$. Therefore,
\begin{eqnarray*}
\frac{d}{dt} (g \cdot \gamma) (t) &=& g_\ast \gamma'(t) \\
&=& g_\ast \left( -\nabla H(\gamma(t), t) \right) \\
&=& -\nabla (H_k \circ g^{-1}) (g \cdot \gamma(t), t) 
\end{eqnarray*}
Finally, $g \cdot \gamma$ satisfies
$$ (g \cdot \gamma) (- \infty) = (gk) \cdot p \in crit(h_{gk}) \qquad (g \cdot \gamma) (+ \infty) = g \cdot p \in crit(h_g).$$
This implies that the continuation map from the homotopy $H_k  \circ g^{-1}$ (from $h_{gk}$ to $h_g$) can be obtained by counting $g \cdot \gamma$'s. Therefore, $g \phi_k g^{-1}$ is the continuation induced by $H_k \circ g^{-1}$.

By composing $g\, \phi_k g^{-1}$ and $\phi_g$, we get another chain map
$$ \phi_g \, g \, \phi_k g^{-1} : CM_\ast (M,h_{gk};\Theta) \to CM_\ast (M,h;\Theta)$$
which is induced by the composition of two homotopies $H_k \circ g^{-1}$ and $H_g$. (Indeed, any composition of two continuation maps again becomes a continuation map by concatenating associated homotopies in this way.)

At the end, we get a chain homotopy $\tilde{\sigma}_{g,k}$ satisfying
$$\phi_g \, g \, \phi_k g^{-1} - \phi_{gh} = \tilde{\sigma}_{g,k} \partial + \partial \tilde{\sigma}_{g,k}.$$ 
since two continuations maps are always chain homotopic. Now, we define $\sigma_{g,k}$ to be the composition $\tilde{\sigma}_{g,k}\, gk$ which gives a desired chain homotopy between $\phi_g \, g \, \phi_k g^{-1} gk = \phi_g \, g \, \phi_k k$ and $\phi_{gk} \, gk$ because $gk : CM_\ast (M, h;\Theta) \to CM_\ast (M, h_{gk};\Theta)$ is a chain map by the definition of $h_{gk}$.
\end{proof}

\subsection{Invariance}
We now consider two Morse-Smale perturbations $h^1$ and $h^2$. By the discussion made in the previous subsection we have weak $G$-action on $CM_\ast (M, h^1;\Theta)$ and $CM_\ast (M,h^2;\Theta)$. We want to show that the homotopy between $h^1$ and $h^2$ (defined similarly to \eqref{continuationhpty}) gives rise to a chain map
$$\psi : CM_\ast (M, h^1;\Theta) \to CM_\ast (M, h^2;\Theta)$$
which is weakly $G$-equivariant in the following sense.

\begin{definition}
Let $C^1_\ast$ and $C^2_\ast$ be chain complexes endowed with weak $G$-actions. A chain map 
$\psi : C^1_\ast \to C^2_\ast$ is called weakly $G$-equivariant if there is a chain homotopy $\tau_g$ on $C^2_\ast$ for each $g \in G$ such that
$$g \cdot_{w} \psi (p) - \psi ( g \cdot_{w} p) = \tau_g \partial + \partial \tau_g.$$
\end{definition}

In particular, a weakly equivariant chain map induces an equivariant homomorphism on the level of homology.

\begin{prop}
$\psi$ above is a $G$-equivariant chain map which induces a $G$-equivariant isomorphism
$$[\psi] : HM_\ast (M, h^1;\Theta) \to HM_\ast (M, h^2;\Theta)$$
\end{prop}

\begin{proof}
We already know that $[\psi]$ is an isomorphism from standard Morse theory. To see the weak equivariance of $\psi$, we have the find a homotopy $\tau_g (p)$ between
$$g \cdot_{w} \psi (p) = \phi_g^2 \, g \, \psi (p) \quad \mbox{and} \quad \psi (g \cdot_{w} p) = \psi \, \phi_g^1 \, g (p).$$
Consider the following commutative diagram for each $g \in G$:
\begin{equation*}
\xymatrix{ CM_\ast (M, h^1_g;\Theta ) \ar[r]^{g \, \psi \, g^{-1}} \ar[d]_{\phi_g^1} & CM_\ast (M, h^2_g;\Theta )\ar[d]^{\phi_g^2} \\
CM_\ast (M, h^1;\Theta) \ar[r]_{\psi} & CM_\ast (M, h^2;\Theta).
}
\end{equation*}
\end{proof}
Let $H$ be the homotopy from $h^1$ to $h^2$ which induces the continuation $\psi$. In the same manner as in the proof of Proposition \ref{weakGonCM}, one can show that $g\, \psi \, g^{-1} : CM_\ast (M, h_g^1;\Theta) \to CM_\ast (M, h_g^2;\Theta)$ is a continuation map induced by the homotopy $H \circ g^{-1}$. 

Since the composition of two continuations maps is again a continuation, both $\phi_g^2 \, g \, \psi  \, g^{-1}$ and $\psi \, \phi_g^1$ are continuation homomorphisms from $CM_\ast (M, h_g^1;\Theta)$ to $CM_\ast (M, h^2;\Theta)$. By the uniqueness of the continuation up to homotopy, there is a chain homotopy $\tilde{\tau}_g$ from $CM_\ast (M, h^2;\Theta)$ to itself satisfying
$$\phi_g^2 \, g \, \psi  \, g^{-1} - \psi \, \phi_g^1 = \tilde{\tau}_g \partial + \partial \tilde{\tau}_g.$$
Therefore, 
$$\phi_g^2 \, g \, \psi  - \psi \, \phi_g^1 \, g= ( \tilde{\tau}_g\, g)\partial + \partial (\tilde{\tau}_g \, g)$$
because $g : CM_\ast (M, h^2;\Theta) \to CM_\ast (M, h^2_g ;\Theta)$ is a chain map by the definition of $h^2_g$. Then, $\tau_g = \tilde{\tau}_g \, g$ gives a desired homotopy.

\bibliographystyle{amsalpha}

\end{document}